\documentclass[10pt]{amsart}
\usepackage{amsmath,amsthm}
\usepackage[colorlinks=true,linkcolor=blue,urlcolor=blue,citecolor=blue]{hyperref}
\usepackage[all]{xy}
\usepackage{graphicx}

\begin{document}

\newtheorem{theorem}{Theorem}[section]
\newtheorem{lemma}[theorem]{Lemma}
\newtheorem{proposition}[theorem]{Proposition}
\newtheorem{corollary}[theorem]{Corollary}

\theoremstyle{definition}
\newtheorem{definition}[theorem]{Definition}
\newtheorem{remark}[theorem]{Remark}
\newtheorem{example}[theorem]{Example}
\newtheorem{problem}[theorem]{Problem}
\newtheorem*{acknowledgements}{Acknowledgements}

\numberwithin{equation}{section}

\renewcommand{\subjclassname}{%
\textup{2020} Mathematics Subject Classification}


\title[Measure data for a general class of nonlinear elliptic problems]{Measure data for a general class of nonlinear elliptic problems}


\author[M. El Ansari]{Mohammed EL ANSARI$^*$}
\address{University of Sidi Mohamed Ben Abdellah, Laboratory of Engineering Sciences, Multidisciplinary Faculty of Taza, P.O. Box 1223 Taza, Morocco.}
\email{$mohammed.elansari1@usmba.ac.ma$}

\author[Y. Akdim]{Youssef AKDIM}
\address{University of Sidi Mohamed Ben Abdellah, LAMA Laboratory, Faculty of Sciences D'har El Mahraz, P.O. Box 1796 Atlas Fez, Morocco.}
\email{$youssef.akdim@usmba.ac.ma$}

\author[S. Lalaoui Rhali]{Soumia LALAOUI RHALI}
\address{University of Sidi Mohamed Ben Abdellah,  Laboratory of Engineering Sciences, Multidisciplinary Faculty of Taza, P.O. Box 1223 Taza, Morocco.}
\email{$soumia.lalaoui@usmba.ac.ma$}


\maketitle

\begin{abstract}
We consider nonlinear elliptic inclusion having a measure in the right-hand side of the type $\beta(u)-div a(x,Du)\ni \mu$  in   $\Omega$ a bounded domain in $\mathbb{R}^{N},$ with $\beta$ is a maximal monotone graph in $\mathbb{R}^2$ and $a(x,Du)$ is a Leray-Lions type operator. We study a suitable notion of solution for this kind of problem. The functional setting involves anisotropic Sobolev spaces.
\end{abstract}

\section{Introduction}

In this paper, we investigate  the existence and uniqueness of an appropriate solution  for  anisotropic elliptic problems with Dirichlet boundary condition
 
\begin{equation*}
 (E,\mu)
   \left\{
\begin{array}{l@{~}l@{~}l@{~}l}
\beta(u)-div a(x,Du)\ni \mu  & \text{in}   &\Omega,\\
u=0 &  \text{on} & \partial\Omega.
\end{array}
 \right.
\end{equation*}

where $\Omega $ is a bounded domain in $\mathbb{R}^{N}(N\geq 1)$ and  $\partial \Omega $
its Lipschitz boundary if $N\geq 2$, a right-hand side $\mu$  is Radon measure.  $\beta :\mathbb{R}%
\rightarrow 2^{\mathbb{R}}$ is a set valued, maximal monotone mapping such
that $0\in \beta (0)$ and 
$a:\Omega \times \mathbb{R^{N}}\rightarrow \mathbb{R^{N}}$ is a
Carath\'{e}odory function satisfying the following conditions, for all $\xi ,\eta \in \mathbb{R}^{N}$ and a.e. in $\Omega$:
$$(\mathbf{H}_{1}) \hspace*{1cm}\sum_{i=1}^{N}a_{i}(x,\xi ).\xi _{i}\geq \lambda \sum_{i=1}^{N}|\xi
_{i}|^{p_{i}}$$

$$(\mathbf{H}_{2}) \hspace*{1cm}|a_i(x,\xi)|\leq \gamma (d_i(x)+|\xi_i^{p_i-1}|) $$

$$(\mathbf{H}_{3}) \hspace*{1cm}(a(x,\xi )-a(x,\eta )).(\xi -\eta )\geq 0,$$

Where $\lambda, \gamma$ are some positive constants, $d_i$ is a positive function in $L^{p^{\prime}_i}(\Omega)$.

Under the above assumptions $(\mathbf{H}_{1})$- $(\mathbf{H}_{3})$, we proved in \cite{ref:L1} the existence and uniqueness of renormalized solutions to $(E, f)$ for  $L^{\infty }-$ data, and in  \cite{ref:L2} for $L^{1}$- data

\begin{equation*}
 (E, f)  \left\{
\begin{array}{l@{~}l@{~}l@{~}l}
\beta(u)-div(a(x,Du)+F(u))\ni f  & \text{in}   &\Omega,\\
u=0 &  \text{on} & \partial\Omega.
\end{array}
 \right.
\end{equation*}

Now, our problem $(E, \mu)$ can be viewed as a generalization of the problem $(E,f)$ such that we  extend the existence results to measure data.

One of the motivations for studying $(E, \mu)$ comes from the role of measures in the study of nonlinear partial differential equations that have become
more and more important in the last years, not only because it belongs to the mathematical spirit to try to extend the scope of a theory, but also because the extension from the function setting to the measure framework appeared in many areas of science, including fluid dynamics \cite{ref:17}, image processing \cite{ref:205}, and mathematical finance.

In dealing with measures data, there are several works in this direction \cite{ref:190, ref:33, ref:192, ref:193, ref:195, ref:196, ref:198}. Boccardo and Gallou\"{e}t proved  the existence of a solution in the sense of distributions of the nonlinear case if we take $\beta\equiv 0$ in $(E,\mu)$:

\begin{equation*}
 (P)
   \left\{
\begin{array}{l@{~}l@{~}l@{~}l}
-div a(x,Du)\ni \mu  & \text{in}   &\Omega,\\
u=0 &  \text{on} & \partial\Omega.
\end{array}
 \right.
\end{equation*}

but this solution is not unique in general. So, it was necessary to find some extra conditions of the distributional solutions of $(P)$ in order to ensure both existence and uniqueness. Three definitions of solutions were introduced: entropy solutions \cite{ref:4}; renormalized solutions \cite{D1} and the notion of solution obtained as limit of approximations \cite{ref:191}. For diffuse measure, that is, for    $ \mu \in L^1(\Omega)+W^{-1, {{p'}}}(\Omega)$, the problem $(P)$ was solved by L. Boccardo, T. Gallou\"{e}t and L. Orsina  in the framework of entropy solutions \cite{ref:192}, and for general measures by G. Dal Maso et al. in \cite{ref:33} in the framework of renormalized solutions. If we take $\beta \not\equiv 0$, $\beta(r)= |r|^{p-1}.r$ and $div a(x, Du)= \Delta u$ in $(E,\mu)$, P. B\'{e}nillan and H. Brezis proved in \cite{ref:194} the existence of weak solutions. Here, we treated the general case when $\beta$ is maximal monotone graph in $\mathbb{R}\times \mathbb{R}$, we recall that this problem has been studied by   N. Igbida et al. in classical Sobolev spaces \cite{ref:195} and F. Andreu et al. treated the case of nonhomogeneous Neumann boundary condition \cite{ref:196}
\begin{equation*}
 (E,\mu, \psi)
   \left\{
\begin{array}{l@{~}l@{~}l@{~}l}
\beta(u)-div(a(x,Du))\ni \mu  & \text{in}   &\Omega,\\
\gamma(u) + a(x, D(u))\ni \psi &  \text{on} & \partial\Omega.
\end{array}
 \right.
\end{equation*}

for any $\mu \in L^{p'}(\Omega)$ and $\psi \in L^{p'}(\partial \Omega)$.

 Our goal here is to extend these results in the framework of anisotropic Sobolev spaces with diffuse measure. There are two major difficulties, ones connected with the domain of the nonlinearity $\beta$ and also with the regularity of the measure $\mu$. We set 
 
 $$
 \operatorname{int}(\operatorname{dom} \beta)=(m, M) \text { with }-\infty \leq m \leq 0 \leq M \leq+\infty.
 $$

The paper is organized as follows: In Section $2$, we recall the standard framework of
anisotropic Sobolev spaces and we proved the definition of elliptic $\overrightarrow{p}$ capacity and some notations which will be used frequently also some technical lemmas. In Section $3,$ we
give our main results on the existence and uniqueness of solutions. We devote Section 4 to the proof of Theorem \ref{1.1}. In section 5, we will prove Theorem \ref{1.2} and Theorem \ref{1.3}.

\section{Preliminaries}

\subsection{Anisotropic Sobolev spaces }
Let $\Omega$ be a bounded open subset of $\mathbb{R}^N(N \geq 2)$,

Let $p_0, p_1, \ldots, p_N$ be $N+1$ exponents, with $1<p_i<\infty$ for $i=0, \ldots, N$. We denote
$$
\vec{p}=\left(p_0, \ldots, p_N\right), \quad D^0 u=u \quad \text { and } \quad D^i u=\frac{\partial u}{\partial x_i} \quad \text { for } \quad i=1, \ldots, N,
$$
and we define
$$
\overline{p}=\min \left\{p_0, p_1, \ldots, p_N\right\} \quad \text { then } \quad \overline{p}>1
$$

The anisotropic  Sobolev spaces  $W^{1, \overrightarrow{p}}(\Omega)$ is defined as follow (see \cite{ref:12})
$$
W^{1, \overrightarrow{p}}(\Omega)=\left\{u \in L^{p_0}(\Omega) \quad \text { and } \quad D^i u \in L^{p_i}(\Omega) \quad \text { for } \quad i=1,2, \ldots, N\right\},
$$
it is a Banach space with respect to norm
$$
\lVert u\lVert_{W^{1,{\overrightarrow{p}}}(\Omega)}=\sum_{i=0}^N\left\|D^i u\right\|_{p_i}
$$

  The space $u \in W^{1,{\overrightarrow{p}}}_0(\Omega)$ is the closure of $C^\infty_0(\Omega)$ with respect to this norm.\\
 The dual space of anisotropic Sobolev space
  $W^{1,{\overrightarrow{p}}}_0(\Omega)$ is equivalent to $W^{-1,{\overrightarrow{p'}}}(\Omega)$, where $\overrightarrow{p'}= (p'_1, ..., p'_N)$ and $p'_i=\dfrac{p_i}{p_i -1}$ for all $i=1, ...,,N$.

  We recall a Poincar\'{e}-type inequality. Let $u \in W^{1,{\overrightarrow{p}}}_0(\Omega)$, then for every $q\geq 1$ there exists a constant $C_p$ (depending on $q$ and $p_i$ (see \cite{ref:16}), such that
   \begin{equation}
     \lVert u\lVert_{L^q(\Omega)} \leq C_p\lVert \partial_{x_i}u\lVert_{L^{p_i}(\Omega)} \hspace*{2mm}for\hspace*{2mm} i=1,...,N.
   \end{equation}
    Moreover a Sobolev-type inequality holds. Let us denote by  $\overline{p}$ the harmonic mean of these numbers, i.e. $\frac{1}{\overline{p}}= \frac{1}{N} \sum_{i=1}^{N} \frac{1}{p_i}$. Let $u\in W^{1,{\overrightarrow{p}}}_0(\Omega)$, then there exists \cite{ref:12} a constant $C_s$ such that

     \begin{equation}
         \lVert u\lVert_{L^q(\Omega)} \leq C_s \prod_{i=1}^{N} \lVert \partial_{x_i}u\lVert_{L^{p_i}(\Omega)}^{\frac{1}{N}} ,
       \end{equation}

\noindent  where   $q= \overline{p}^*= \frac{N \overline{p}}{N-
\overline{p}}$ if $\overline{p}< N$ or  $q\in [1, +\infty[$ if
$\overline{p} \geq N.$ On the right-hand side of $(2.2)$ it is
possible to replace the geometric mean by the arithmetic mean: let
$a_1,...,a_N$ be positive numbers, it holds
   \begin{equation*}
   \prod_{i=1}^{N} a_i^{\frac{^1}{N}}\leq \dfrac{1}{N} \sum_{i=1}^{N}a_i,
   \end{equation*}

 \noindent which implies by $(2.2)$ that

   \begin{equation}
          \lVert u\lVert_{L^q(\Omega)} \leq \dfrac{C_s}{N} \sum_{i=1}^{N} \lVert \partial_{x_i}u\lVert_{L^{p_i}(\Omega)} .
        \end{equation}

  Note that when the following inequality holds  \begin{equation}
          \overline{p}< N,
        \end{equation}

 \noindent inequality $(2.3)$ implies the continuous embedding of the space $W^{1,{\overrightarrow{p}}}_0(\Omega)$ into $L^q(\Omega)$ for every $q\in [1,\overline{p}^*].$ On the other hand, the continuity of the embeding $W^{1,{\overrightarrow{p}}}_0(\Omega)\hookrightarrow L^{p_+}(\Omega)$ with $p_{+}:= max\{p_1,...,p_N\}$ relies on inequality $(2.1)$. It may happen that $\overline{p}^* <p_+$ if the exponent $p_i$ are closed enough, then $p_\infty:= max\{\overline{p}^*, p_+\}$ turns out to be the critical exponent in the anisotrpic Sobolev embedding (see \cite{ref:12}).
\begin{proposition}
If the condition $(2.4)$ holds, then for $q\in [1,p_\infty]$ there is a continuous embedding $W^{1,{\overrightarrow{p}}}_0(\Omega)\hookrightarrow L^{q}(\Omega)$. For $q<p_\infty$ the embedding is compact.
  \begin{equation}
  W^{1,{\overrightarrow{p}}}_0(\Omega)\hookrightarrow \hookrightarrow L^{q}(\Omega).
  \end{equation}
\end{proposition}

\subsection{Measures}

We start recalling the definition of ${\overrightarrow{p}}$-capacity \cite{ref:197}.

 The ${\overrightarrow{p}}$-capacity cap ${ }_{\overrightarrow{p}}(E, \Omega)$ of a compact set $E \subset \Omega$ with respect to $\Omega$ is:
$$
\operatorname{cap}_{\overrightarrow{p}}(E, \Omega)=\inf \left\{\int_{\Omega}|\nabla \varphi|^{\overrightarrow{p}}: \varphi \in  W^{1,{\overrightarrow{p}}}_0(\Omega), \varphi \geq \chi_E\right\},
$$
where $\chi_E$ is the characteristic function of $E$ (we will use the convention that $\inf \emptyset=+\infty$ ). 

If $U \subseteq \Omega$ is an open set, then we denote
$$
\operatorname{cap}_{\overrightarrow{p}}(U, \Omega)=\sup \left\{\operatorname{cap}_{\overrightarrow{p}}(E, \Omega): E \text { compact, } E \subseteq U\right\},
$$
while the $p$-capacity of any subset $B \subseteq \Omega$ is defined as:
$$
\operatorname{cap}_{\overrightarrow{p}}(B, \Omega)=\inf \left\{\operatorname{cap}_{\overrightarrow{p}}(U, \Omega): U \text { open, } B \subseteq U\right\}
$$

The Lebesgue measure of a Borel set $E \subseteq \mathbb{R}^N$ will be denoted by $\mathcal{L}^N(E)$. We denote by $\mathcal{M}_{{b}}(\Omega)$ the space of all Radon measures on $\Omega$ with bounded total variation, and $C_{{b}}(\Omega)$ as the space of all bounded, continuous functions on $\Omega$, so that $\int_{\Omega} \varphi \mathrm{d} \mu$ is defined for $\varphi \in C_{{b}}(\Omega)$ and $\mu \in \mathcal{M}_{{b}}(\Omega)$.

The positive part, the negative part, and the total variation of a measure $\mu$ in $\mathcal{M}_{\mathrm{b}}(\Omega)$ are denoted by $\mu^{+}, \mu^{-}$, and $|\mu|$, respectively.

We recall that for a measure $\mu$ in $\mathcal{M}_{{b}}(\Omega)$, and a Borel set $E \subseteq \Omega$, the restriction of $\mu$ in $E$ is the measure $\mu (\mu\lfloor E) E$ defined by $(\mu \lfloor E)(B)=\mu(E \cap B)$ for any Borel set $B \subseteq \Omega$.

We denote by $\mathcal{M}_0(\Omega)$ the set of all measures $\mu$ in $\mathcal{M}_{{b}}(\Omega)$ which satisfy $\mu(B)=0$ for every Borel set $B \subseteq \Omega$ such that $\operatorname{cap}_{\overrightarrow{p}}(B, \Omega)=0$, while $\mathcal{M}_{\mathrm{s}}(\Omega)$ will be the set of all measures $\mu$ in $\mathcal{M}_{\mathrm{b}}(\Omega)$ for which there exists a Borel set $E \subset \Omega$, with $\operatorname{cap}_{\overrightarrow{p}}(E, \Omega)=0$, and such that $\mu=\mu \lfloor E$.

We give now some results, well known for bounded open sets $\Omega$, which can be easily extended to the case of unbounded domains.

\begin{lemma}
A measure $\mu_0$ belongs to $\mathcal{M}_0(\Omega)$ if and only if it belongs to the space $L^1(\Omega)+W^{-1, {\overrightarrow{p'}}}(\Omega)$, that is there exist $f \in L^1(\Omega)$, and $F \in\left(L^{{\overrightarrow{p'}}}(\Omega)\right)^N$ such that

$$
\mu_0=f-\operatorname{div} F .
$$

Moreover

$$
\int_{\Omega} v \mathrm{~d} \mu_0=\int_{\Omega} v f \mathrm{~d} x + \int_{\Omega} D v \cdot F \mathrm{~d} x,
$$

for every $v \in W_0^{1, {\overrightarrow{p}}}(\Omega) \cap L^{\infty}(\Omega)$.
\end{lemma}

\begin{proof}

 Up to minor changes, the proof is the same as that proposed in Theorem 2.1 in \cite{ref:192} for bounded domains.

\end{proof}

As regards a general measure $\mu \in \mathcal{M}_{{b}}(\Omega)$, the following decomposition result holds.

\begin{theorem}
For every $\mu \in \mathcal{M}_{{b}}(\Omega)$ there exists a unique pair $\left(\mu_0, \mu_{\mathrm{s}}\right)$ such that $\mu=\mu_0+\mu_{\mathrm{s}},$  $\mu_0 \in \mathcal{M}_0(\Omega)$, $\mu_{\mathrm{s}} \in \mathcal{M}_{\mathrm{s}}(\Omega)$.

\end{theorem}

\begin{proof} 
See \cite{ref:198}, Lemma 2.1.
\end{proof}

We denote in all the rest of this work $\mathcal{M}_0(\Omega)$ by $\mathcal{M}_b^{\overrightarrow{p}}(\Omega)$
 that is called diffuse measure.

\subsection{Notation and functions}

We denote by $\mathcal{T}_0^{1, {\overrightarrow{p}}}(\Omega)$, the space of measurable functions $u: \Omega \longrightarrow \mathbb{R}$ such that for any $k>0$. $T_k(u) \in W_0^{1, {\overrightarrow{p}}}(\Omega)$.

Where  $T_k : \mathbb{R}\rightarrow \mathbb{R}$ is  the truncature operator that define as

\begin{equation*}
 \left\{
 \begin{array}{r@{~}c@{~}l}
 -k, \text{if} &r\leq -k,& \\ r, \text{if} &|r|<k,&\\ k, \text{if} &r\geq k.&
 \end{array}
 \right.
 \end{equation*}

  \noindent And for $r \in \mathbb{R},$ let $r\rightarrow r^+:= max(r,0), r \rightarrow sign_0(r)$ the usual sign function which is equal

  \begin{equation*}
  r \rightarrow sign_0(r):=
  \left\{
  \begin{array}{r@{~}c@{~}l}
  -1, &\text{on} &]-\infty,0[, \\ 1, &\text{on} &]0,\infty[,\\ 0,  &\text{if} &r=0 .
  \end{array}
  \right.
  \end{equation*}

  \begin{equation*}
   r \rightarrow sign_0^+(r):=
   \left\{
   \begin{array}{r@{~}c@{~}l}
   1, &\text{if} &r> 0, \\ 0, &\text{if} & r\leq 0.
   \end{array}
   \right.
   \end{equation*}

For any $r \in \mathbb{R}$ and any measurable function $u$ on $\Omega,[u=r],[u \leq r]$ and $[u \geq r]$ denote respectively the sets $\{x \in \Omega: u(x)=r\},\{x \in \Omega: u(x) \leq r\}$ and $\{x \in \Omega: u(x) \geq r\}$.

 Now, let us prove the following result which will be useful in the sequel.
  
  \begin{lemma}\label{L2.3}
   Let $\left(\beta_n\right)_{n \geq 1}$ be a sequence of maximal monotone graphs such that $\beta_n \rightarrow \beta$ in the sense of graphs. We consider $\left(z_n\right)_{n \geq 1}$ and $\left(w_n\right)_{n \geq 1}$ two sequences of $L^1(\Omega)$, such that $w_n \in \beta_n\left(z_n\right)$. $\mathcal{L}^N$-c.e. in $\Omega$, for any $n \in \mathbb{N}^*$. If
  $$
  \left(w_n\right)_{n \geq 1} \text { is bounded in } L^1(\Omega) \text { and } z_n \longrightarrow z \text { in } L^1(\Omega) \text {, }
  $$
  then
  $$
  z \in \operatorname{dom}(\beta) \quad \mathcal{L}^N \text {-a.e. in } \Omega \text {. }
  $$
  
  \end{lemma}
  
  The main tool for the proof of Lemma \ref{L2.3} is the "biting lemma of Chacon" (see \cite{ref:199}). Let us recall it.
  
  \begin{lemma}\label{L2.4} (The "biting lemma of Chocon"). Let $\Omega$ be an open bounded subset of $\mathbb{R}^N$ and $\left(f_n\right)_n$ a bounded sequence in $L^1(\Omega)$. Then, there exist $f \in L^1(\Omega)$, a subsequence $\left(f_{n_k}\right)_k$ and a sequence of measurable sets $\left(E_j\right)_j$. $E_j \subset \Omega, \forall j \in \mathbb{N}$ with $E_{j+1} \subset E_j$ and $\lim _{j \rightarrow+\infty}\left|E_j\right|=0$, such that for any $j \in \mathbb{N}, f_{n_k} \rightarrow f$ in $L^1\left(\Omega \backslash E_j\right)$.
  \end{lemma}
  \begin{proof}[Proof of Lemma \ref{L2.3}.]
  Since the sequence $\left(w_n\right)_{n \geq 1}$ is bounded in $L^1(\Omega)$. using the "biting lemma of Chacon" there exist $w \in L^1(\Omega)$, a subsequence $\left(w_{n_k}\right)_{k \geq 1}$ and a sequence of measurable sets $\left(E_j\right)_{j \in \mathbb{N}}$ in $\Omega$ such that $E_{j+1} \subset E_j, \forall j \in \mathbb{N}, \lim _{j \rightarrow+\infty}\left|E_j\right|=0$ and $\forall j \in \mathbb{N}, w_{n_k} \rightarrow w$ in $L^1\left(\Omega \backslash E_j\right)$. Since $z_{n_k} \longrightarrow z$ in $L^1(\Omega)$ and so in $L^1\left(\Omega \backslash E_j\right), \forall j \in \mathbb{N}$ and $\beta_{n_k} \longrightarrow \beta$ in the sense of graphs, we have $w \in \beta(z)$ a.e. in $\Omega \backslash E_j$. Thus $z \in \operatorname{dom}(\beta)$ a.e. in $\Omega \backslash E_j$. Finally we obtain $z \in \operatorname{dom}(\beta)$ a.e. in $\Omega$.
  \end{proof}

\section{Main result}

  \begin{theorem}\label{1.1}
   For any $\mu \in \mathcal{M}_b^{\overrightarrow{p}}(\Omega)$, the problem $(E,\mu)$ has at least one solution $(u, w)$ in the sense that:\\

   \begin{itemize}

     \item[$\bullet$]  $w \in L^1(\Omega), u$ is measurable, $u \in \operatorname{dom}(\beta)$   $\mathcal{L}^N$-a.e. in $\Omega, T_k(u) \in W_0^{1, {\overrightarrow{p}}}(\Omega), \forall k>0, w \in \beta(u) \mathcal{L}^N$-a.e. in $\Omega$,

      \item[$\bullet$]  there exists a measure $v \in \mathcal{M}_b(\Omega)$ such that $v \perp \mathcal{L}^N$, for any $h \in\mathcal{C}_c(\mathbb{R}), h(u) \in L^{\infty}(\Omega, d|v|)$, $h(u) v \in M_b^{\overrightarrow{p}}(\Omega)$,
      \begin{center}

       $v^{+}$is concentred on $[u=M] \cap[u \neq+\infty]$
      
       $\quad v^{-}$is concentred on $[u=m] \cap[u \neq-\infty]$
       
        \end{center}
      \begin{equation}\label{1.6}
      \int_{\Omega} a(x, D u) \cdot D(h(u) \xi) d x+\int_{\Omega} w h(u) \xi d x+\int_{\Omega} h(u) \xi d v=\int_{\Omega} h(u) \xi d \mu,
      \end{equation}
      
     \item[$\bullet$] for any $\xi \in W_0^{1, {\overrightarrow{p}}}(\Omega) \cap L^{\infty}(\Omega)$ and
      \begin{equation}\label{1.7}
      \lim _{n \rightarrow+\infty} \sum_{i=1}^{N} \int_{[n \leq |u| \leq n+1]}|D u|^{{p_i}} d x=0
      \end{equation}

     \end{itemize}

     \end{theorem}

     \begin{remark}

  \item[$\ast$] Since $T_k(u) \in W_0^{1, {\overrightarrow{p}}}(\Omega)$, without abusing we are identifying the function $u$ with its quasicontinuous representation. So, since the measures $\mu$ and $v$ are diffuse, all the terms of Theorem \ref{1.1} have a sense.
  
  \item[$\ast$] If $M=+\infty$ (resp. $m=-\infty$ ), then $v^{+} \equiv 0$ (resp. $v^{-} \equiv 0$ ). In the particular case where the domain of $\beta$ is equal to $\mathbb{R}$, Theorem \ref{1.1} implies the existence of a renormalized solution in the standard sense. More precisely, we have
  
    \end{remark}
  
  \begin{corollary}\label{C1.1}
  Assume that
  $$
  \mathcal{D}(\beta)=\mathbb{R},
  $$
  for any $\mu \in \mathcal{M}_b^{\overrightarrow{p}}(\Omega)$, the problem $(E,\mu)$ hos at least one solution $(u, w)$ in the sense that $w \in L^1(\Omega), u$ is meosurable, $T_k(u) \in W_0^{1, {\overrightarrow{p}}}(\Omega), \forall k>0, w \in \beta(u) \mathcal{L}^N$-c.e. in $\Omega$, for any $h \in \mathcal{C}_c(\mathbb{R})$, we have
  \begin{equation}\label{1.8}
  \int_{\Omega} a(x, D u) \cdot D(h(u) \xi) d x+\int_{\Omega} w h(u) \xi d x=\int_{\Omega} h(u) \xi d \mu,
  \end{equation}
  
  for any $\xi \in W_0^{1, {\overrightarrow{p}}}(\Omega) \cap L^{\infty}(\Omega)$ and
  \begin{equation}\label{1.9}
  \lim _{n \rightarrow+\infty}\sum_{i=1}^{N} \int_{[n \leq|u| \leq n+1]}|D u|^{{p_i}} d x=0.
  \end{equation}

  \end{corollary}
  
The proof of uniqueness follows by the using the entropic formulation of the solution.

  \begin{theorem}\label{1.2}
  If $(u, w)$ is a solution of $(E,\mu)$, then $(u, w)$ is a solution in the following sense: for any $\xi \in W_0^{1, {\overrightarrow{p}}}(\Omega) \cap L^{\infty}(\Omega)$ such that $\xi \in \operatorname{dom} \beta$.
  \begin{equation}\label{1.10}
  \int_{\Omega} a(x, D u) \cdot D T_k(u-\xi) d x+\int_{\Omega} w T_k(u-\xi) d x \leq \int_{\Omega} T_k(u-\xi) d \mu, \quad \text { for any } k>0 .
  \end{equation}
  
  \end{theorem}

  In the case where the domain of $\beta$ is bounded, we believe that it could be also useful for the proof of uniqueness in a more general setting like those of \cite{ref:33} and \cite{ref:193}. And the renormalization with the function $h$ is not necessary in Theorem \ref{1.1}. We can take $h \equiv 1$. Moreover, by using Theorem \ref{1.2} we have uniqueness. This is summarize in the following theorem.

  \begin{theorem}\label{1.3}
  If $-\infty<m \leq 0 \leq M<\infty$, then, for any $\mu \in \mathcal{M}_b^{\overrightarrow{p}}(\Omega)$, the problem $(E,\mu)$ has a unique solution $(u, w)$ in the sense that $(u, w) \in W_0^{1, {\overrightarrow{p}}}(\Omega) \times L^1(\Omega), u \in \operatorname{dom}(\beta)$     $\mathcal{L}^N$-a.e. in $\Omega, w \in \beta(u)$  $\mathcal{L}^N$-a.e. in $\Omega$, there exists $v \in \mathcal{M}_b^{{\overrightarrow{p}}}(\Omega)$ such that $v \perp \mathcal{L}^N$.
  $$
  v^{+} \text {is concentred on }[u=M], \quad v^{-} \text {is concentred on }[u=m]
  $$
  and
  \begin{equation}\label{1.11}
  \int_{\Omega} a(x, D u) \cdot D \xi d x+\int_{\Omega} w \xi d x+\int_{\Omega} \xi d v=\int_{\Omega} \xi d \mu,
  \end{equation}

  for any $\xi \in W_0^{1, {\overrightarrow{p}}}(\Omega) \cap L^{\infty}(\Omega)$. Moreover

  \begin{equation}\label{1.12}
  v^{+} \leq \mu_s \lfloor [u=M]
  \end{equation}
  
  and

  \begin{equation}\label{1.13}
  v^{-} \leq -\mu_s \lfloor [u=m].
  \end{equation}
  
  \end{theorem}

  \begin{remark}\label{R1.1}

  1. If $-\infty<m \leq 0 \leq M<\infty$, for any $\mu \in \mathcal{M}_b^{\overrightarrow{p}}(\Omega)$, the good measure with respect to the notion of weak solution associated with $\mu$ is given by
  $$
  \mu^*=\mu-v .
  $$
  2. Assuming that $M<\infty$ (resp. $-\infty<m$ ) and $\overline{\mathcal{D}(\beta)}=(-\infty, M]$ (resp. $\overline{\mathcal{D}(\beta)}=[m, \infty)$ ). we can prove also that, if $(u, w)$ is a solution in the sense of Theorem 1.1, then
  $$
  v^{+} \leq \mu_s\left\lfloor[u=M] \text { and } v^{-} \equiv 0\right.
  $$
  (resp.
  $$
  v^{-} \leq-\mu_s\left\lfloor[u=m] \text { and } v^{+} \equiv 0\right) .
  $$
  3. If the measure $\mu$ is regular (i.e. absolutely continuous with respect to Lebesgue measure), Theorem \ref{1.3} and the previous Remark show that $\nu=0$ and a solution in the sense of Theorem \ref{1.1} coincides with the usual renormalized solution for $(E,\mu)$, which corresponds to the unique weak solution in the case where $\mathcal{D}(\beta)$ is bounded.
  
  \end{remark}

  Notice that this kind of formulation for the solution has already appeared in previous work, for instance in \cite{ref:202, ref:203} to deal with nonlinearities $\beta$ depending on $x$. It appeared also in \cite{ref:204} to treat the obstacle problem; i.e. the case where $\mathcal{D}(\beta)=[\mathrm{m}, M]$. Our results here are some kind of generalization of the last result to general nonlinearity $\beta$ and Leray-Lions operator $a$ \cite{ref:15}.

  \section{ Proof of Theorem \ref{1.1}}

  If $(u, w)$ is a solution of $(E,\mu)$, choosing $\xi=T_k(u), k>0$ in (\ref{1.11}), we get the following estimate:
    \begin{equation}\label{2.1}
    \forall k>0, \quad \frac{1}{k} \sum_{i=1}^{N} \int_{[|u|<k]}|D u|^{{p_i}} d x \leq K,
    \end{equation}
    
    with $0<K<+\infty$.
    
     The proof of the following two lemmas may be found in \cite{ref:L2, ref:4}.
    
    \begin{lemma}\label{L2.1}
    
    Let $1<{\overrightarrow{p}}<N, \Omega$ be as above ond let $u \in \mathcal{T}_0^{1, {\overrightarrow{p}}}(\Omega)$ be such that (\ref{2.1}) holds. Then there exists $C=C(N, p)>0$ such that
    \begin{equation}\label{2.2}
    \operatorname{meas}([|u|>k]) \leq C K^{\frac{N}{N-\overline{p}}} k^{-p_1}
    \end{equation}
    with $p_1=\frac{N(\overline{p}-1)}{N-\overline{p}}$.
    
    \end{lemma}

    \begin{lemma}\label{L2.2}
    
     Let $1<{\overrightarrow{p}}<N$ and assume that $u \in \mathcal{T}_0^{1, {\overrightarrow{p}}}(\Omega)$ satisfies (\ref{2.1}) for every $k>0$. Then for every $h>0$ we hove
    \begin{equation}\label{2.3}
    \operatorname{meas}([|D u|>h]) \leq C(N, p_i) K \frac{N}{N-1} h^{-p_2}
    \end{equation}
    with $p_2=\frac{N(\overline{p}-1)}{N-1}$.
    
    \end{lemma}

  For every $\epsilon>0$, we consider the Yosida regularization $\beta_\epsilon$ of $\beta$ given by
  $$
  \beta_\epsilon=\frac{1}{\epsilon}\left(I-(I+\epsilon \beta)^{-1}\right) .
  $$
  Thanks to \cite{ref:8}, there exists $j$ a nonnegative, convex and l.s.c. function defined on $\mathbb{R}$, such that
  $$
  \beta=\partial j
  $$
  To regularize $\beta$, we consider
  $$
  j_\epsilon(s)=\min _{r \in \mathbb{R}}\left\{\frac{1}{2 \epsilon}|s-r|^2+j(r)\right\}, \quad \forall s \in \mathbb{R}, \forall \epsilon>0 .
  $$
  By Proposition 2.11 in \cite{ref:8} we have

  $$
  \left\{\begin{array}{l}
  \operatorname{dom}(\beta) \subset \operatorname{dom}(j) \subset \overline{\operatorname{dom}(j)}=\overline{\operatorname{dom}(\beta)}, \\
  j_\epsilon(s)=\frac{\epsilon}{2}\left|\beta_\epsilon(s)\right|^2+j\left(J_\epsilon(s)\right) \text { where } J_\epsilon:=(I+\epsilon \beta)^{-1}, \\
  j_\epsilon \text { is a convex, Frechet-differentiable function and } \beta_\epsilon=\partial j_\epsilon, \\
  j_\epsilon \uparrow j \text { as } \epsilon \downarrow 0 .
  \end{array}\right.
  $$
  Moreover, for any $\epsilon>0, \beta_\epsilon$ is a nondecreasing and Lipschitz-continuous function.
  For any measure $\mu$ assumed diffuse with respect to the capacity $W_0^{1, \overrightarrow{p}}(\Omega)$, a well known result in \cite{ref:192} allows us to write
  \begin{equation}\label{3.1}
  \mu=f-\nabla \cdot F
  \end{equation}
  
  where $f \in L^1(\Omega)$ and $F \in\left(L^{\overrightarrow{p'}}(\Omega)\right)^N$. To regularize $\mu$, for any $\epsilon>0$, we define the functions
  $$
  f_\epsilon(x)=T_{\frac{1}{\epsilon}}(f(x)) \text { for any } x \in \Omega
  $$
  and
  $$
  \mu_\epsilon=f_\epsilon-\nabla \cdot F \text { for any } \epsilon>0 .
  $$
  Then, we consider the following approximating scheme problem
  $$
  P_\epsilon\left(\beta_\epsilon, \mu_\epsilon\right) \begin{cases}- D \cdot a\left(x, D u_\epsilon\right)+\beta_\epsilon\left(u_\epsilon\right)=\mu_\epsilon & \text { in } \Omega, \\ u_\epsilon=0 & \text { on } \partial \Omega .\end{cases}
  $$
  Thanks to \cite{ref:200}, we know that $(E, \mu_\epsilon)$ admits a unique weak solution $u_\epsilon$ in the sense that $u_\epsilon \in$ $W_0^{1, {\overrightarrow{p}}}(\Omega), \beta_\epsilon\left(U_\epsilon\right) \in L^1(\Omega)$ and $\forall \varphi \in W_0^{1, {\overrightarrow{p}}}(\Omega) \cap L^{\infty}(\Omega)$.
  \begin{equation}\label{3.2}
  \int_{\Omega} a\left(x, D u_\epsilon\right) \cdot D \varphi d x+\int_{\Omega} \beta_\epsilon\left(u_\epsilon\right) \varphi d x=\int_{\Omega} f_\epsilon \varphi d x+\int_{\Omega} F \cdot D \varphi d x .
  \end{equation}
  Let us prove the following result.
  
  \begin{proposition}\label{P3.1}

  \begin{itemize}
  \item[(i)] There exists $0<\mathrm{C}<+\infty$ such that for any $k>0$,
  \begin{equation}\label{3.3}
  \sum_{i=1}^{N}\int_{[\left| u_\epsilon \mid \leq k\right]}\left|D u_\epsilon\right|^{p_i} d x \leq C k
  \end{equation}
  \item[(ii)] The sequence $\left(\beta_\epsilon\left(u_\epsilon\right)\right)_{\epsilon>0}$ is uniformly bounded in $L^1(\Omega)$.
  \item[(iii)] For any $k>0$, the sequence $\left(\beta_\epsilon\left(T_k\left(u_\epsilon\right)\right)\right)_{\epsilon>0}$ is uniformly bounded in $L^1(\Omega)$.
  \item[(iv)] There exists $u \in \mathcal{T}_0^{1, {\overrightarrow{p}}}(\Omega)$ such that $u \in \operatorname{dom}(\beta)$ a.e. in $\Omega$ and
  
  \begin{equation}\label{3.4}
  u_\epsilon \longrightarrow u \text { in measure and a.e. in } \Omega, \text { os } \epsilon \longrightarrow 0 \text {. }
  \end{equation}
  
  \end{itemize}
  \end{proposition}

  \begin{proof}

   (i) For any $k>0$, we take $\varphi=T_k\left(u_\epsilon\right)$ as test function in (\ref{3.2}). We get
   
  \begin{equation}\label{3.5}
  \int_{\Omega} a\left(x, D u_\epsilon\right) \cdot D T_k\left(u_\epsilon\right) d x+\int_{\Omega} \beta_\epsilon\left(u_\epsilon\right) T_k\left(u_\epsilon\right) d x=\int_{\Omega} f_\epsilon T_k\left(u_\epsilon\right) d x+\int_{\Omega} F \cdot D T_k\left(u_\epsilon\right) d x .
  \end{equation}
  
  Since
  $$
  \left|\int_{\Omega} f_\epsilon T_k\left(u_\epsilon\right) d x+\int_{\Omega} F \cdot D T_k\left(u_\epsilon\right) d x\right|=\left|\int_{\Omega} T_k\left(u_\epsilon\right) d \mu_\epsilon\right| \leq k|\mu|(\Omega) \leq C k
  $$
  and
  $$
  \int_{\Omega} \beta_\epsilon\left(u_\epsilon\right) T_k\left(u_\epsilon\right) d x \geq 0
  $$
  we deduce that
  $$
  \int_{\Omega} a\left(x, D u_\epsilon\right) \cdot D T_k\left(u_\epsilon\right) d x \leq C k
  $$
  Using $(\mathbf{H}_{1})$, we obtain $\lambda \sum_{i=1}^{N}\int_{\Omega}\left|D T_k\left(u_\epsilon\right)\right|^{p_i} d x \leq C k$ and (i) follows.
  
  (ii) For any $k>0$, the first term of \ref{3.5} is nonnegative, then it follows that
  $$
  \int_{\Omega} \beta_\epsilon\left(u_\epsilon\right) T_k\left(u_\epsilon\right) d x \leq k|\mu|(\Omega) \leq C k
  $$
  Dividing by $k$, we get
  $$
  \frac{1}{k} \int_{\Omega} \beta_\epsilon\left(u_\epsilon\right) T_k\left(u_\epsilon\right) d x \leq C .
  $$
  Letting $k$ go to 0 , we deduce from the inequality above
  $$
  \int_{\Omega} \beta_\epsilon\left(u_\epsilon\right) \operatorname{sign}_0\left(u_\epsilon\right) d x \leq C,
  $$
  which implies $\int_{\Omega}\left|\beta_\epsilon\left(u_\epsilon\right)\right| d x \leq C$ and so $\left(\beta_\epsilon\left(u_\epsilon\right)\right)_\epsilon$ is bounded in $L^1(\Omega)$.
  
  (iii) Since
  $$
  \int_{\Omega}\left|\beta_\epsilon\left(T_k\left(u_\epsilon\right)\right)\right| d x \leq \int_{\Omega}\left|\beta_\epsilon\left(u_\epsilon\right)\right| d x
  $$
  (iii) follows obviously from (ii).
  
  (iv) Using (i) we can assert that for all $k>0$, the sequence $\left(D T_k\left(u_\epsilon\right)\right)_{\epsilon>0}$ is bounded in $L^{\overrightarrow{p}}(\Omega)$. thus the sequence $\left(T_k\left(u_\epsilon\right)\right)_{\epsilon>0}$ is bounded in $W_0^{1, {\overrightarrow{p}}}(\Omega)$. Then up to a subsequence, we can assume that as $\epsilon \longrightarrow 0,\left(T_k\left(u_\epsilon\right)\right)_{\epsilon>0}$ converges strongly to some function $\sigma_k$ in $L^q(\Omega)$ and a.e. in $\Omega$ for any $1 \leq q<\overline{p}^{*}=\frac{N \overline{p}}{N-\overline{p}}$. Let us see that the sequence $\left(u_\epsilon\right)_{\epsilon>0}$ is Cauchy in measure.
  
  Indeed let $s>0$ and define
  $$
  E_1:=\left[\left|u_{\epsilon_1}\right|>k\right], \quad E_2:=\left[\left|u_{\epsilon_2}\right|>k\right] \quad \text { and } E_3:=\left[\left|T_k\left(u_{\epsilon_1}\right)-T_k\left(u_{\epsilon_2}\right)\right|>s\right]
  $$
  where $k>0$ is to be fixed. We note that
  $$
  \left[\left|u_{\epsilon_1}-u_{\epsilon_2}\right|>s\right] \subset E_1 \cup E_2 \cup E_3
  $$
  and hence
  \begin{equation}\label{3.6}
  \operatorname{meas}\left(\left[\left|u_{\epsilon_1}-u_{\epsilon_2}\right|>s\right]\right) \leq \operatorname{meas}\left(E_1\right)+\operatorname{meas}\left(E_2\right)+\operatorname{meas}\left(E_3\right) .
  \end{equation}
  
  Let $\theta>0$. Using Lemma \ref{L2.1}, we choose $k=k(\theta)$ such that
  \begin{equation}\label{3.7}
  \text { meas }\left(E_1\right) \leq \theta / 3 \text { and meas }\left(E_2\right) \leq \theta / 3 \text {. }
  \end{equation}
  Since $\left(T_k\left(u_\epsilon\right)\right)_{\epsilon>0}$ converges strongly in $L^q(\Omega)$ then it is a Cauchy sequence in $L^q(\Omega)$.
  Thus
  \begin{equation}\label{3.8}
  \operatorname{meas}\left(E_3\right) \leq\frac{1}{s^q} \int_{\Omega}\left|T_k\left(u_{\epsilon_1}\right)-T_k\left(u_{\epsilon_2}\right)\right|^q d x \leq \frac{\theta}{3},
  \end{equation}
  for all $\epsilon_1, \epsilon_2 \geq n_0(s, \theta)$. Finally from (\ref{3.6}). (\ref{3.7}) and (\ref{3.8}) we obtain
  \begin{equation}\label{3.9}
  \text { meas }\left(\left[\left|u_{\epsilon_1}-u_{\epsilon_2}\right|>s\right]\right) \leq \theta \text { for all } \epsilon_1, \epsilon_2 \geq n_0(s, \theta) .
  \end{equation}
  Relation (\ref{3.9}) means that the sequence $\left(u_\epsilon\right)_{\epsilon>0}$ is Cauchy in measure, so $u_\epsilon \longrightarrow u$ in measure and up to a subsequence, we have $u_\epsilon \longrightarrow u$ a.e. in $\Omega$. Hence $\sigma_k=T_k(u)$ a.e. in $\Omega$ and so $u \in \mathcal{T}_0^{1, {\overrightarrow{p}}}(\Omega)$. Using Lemma \ref{L2.3}, we have $T_k(u) \in \operatorname{dom} \beta$ a.e. in $\Omega$ for any $k>0$. This implies that $u \in \operatorname{dom} \beta$ a.e. in $\Omega$.

  \end{proof}

  \begin{lemma}\label{P3.2}
  
  For any $k>0$, as $\in$ tends to 0 , we have:
  \begin{itemize}

  \item[(i)] $a\left(x, D T_k\left(u_\epsilon\right)\right) \rightharpoonup a\left(x, D T_k(u)\right)$ weakly in $\prod_{i=1}^{N}L^{p_{i}^{^{%
  \prime }}}(\Omega )$.
  \item[(ii)] $D T_k\left(u_\epsilon\right) \longrightarrow D T_k(u)$ a.e. in $\Omega$.
  \item[(iii)] $a\left(x, D T_k\left(u_\epsilon\right)\right) \cdot D T_k\left(u_\epsilon\right) \longrightarrow a\left(x, D T_k(u)\right) \cdot D T_k(u)$ a.e. in $\Omega$ and strongy in $L^1(\Omega)$.
  \item[(iv)] $D T_k\left(u_\epsilon\right) \longrightarrow D T_k(u)$ strongly in $\left(L^{\overrightarrow{p}}(\Omega)\right)^N$.
  \end{itemize}
  \end{lemma}
   
  \begin{proof}
  
   (i) Using $(\mathbf{H}_{2})$ we see that the sequence $\left(a\left(x, D T_k\left(u_\epsilon\right)\right)\right)_{\epsilon>0}$ is bounded in $\left(L^{{\overrightarrow{p'}}}(\Omega)\right)^N$, then up to a subsequence
  $$
  a\left(x, D T_k\left(u_\epsilon\right)\right) \rightharpoonup H_k \quad \text { in }\prod_{i=1}^{N}L^{p_{i}^{^{%
  \prime }}}(\Omega ) .
  $$
  Let us prove that $H_k=a\left(x, D T_k(u)\right)$ a.e. in $\Omega$ with the same arguments  as in the proof in Lemma 6.1 in \cite{ref:L2}. The proof consists in four steps.
  
  Step 1: For every function $h \in W^{1,+\infty}(\mathbb{R}), h \geq 0$ with supp( $h$ ) compact.
  \begin{equation}\label{3.10}
  \limsup _{\epsilon \rightarrow 0} \int_{\Omega} a\left(x, D u_\epsilon\right) \cdot D\left[h\left(u_\epsilon\right)\left(T_k\left(u_\epsilon\right)-T_k(u)\right)\right] d x \leq 0 .
  \end{equation}
  
  Taking $h\left(u_\epsilon\right)\left(T_k\left(u_\epsilon\right)-T_k(u)\right)$ as test function in (\ref{3.2}), we have

  \begin{equation}\label{3.11}
  \begin{gathered}
  \int_{\Omega} a\left(x, D u_\epsilon\right) \cdot D\left[h\left(u_\epsilon\right)\left(T_k\left(u_\epsilon\right)-T_k(u)\right)\right] d x+\int_{\Omega} \beta_\epsilon\left(u_\epsilon\right) h\left(u_\epsilon\right)\left(T_k\left(u_\epsilon\right)-T_k(u)\right) d x \\
  =\int_{\Omega} f_\epsilon h\left(u_\epsilon\right)\left(T_k\left(u_\epsilon\right)-T_k(u)\right) d x+\int_{\Omega} F \cdot D\left[h\left(u_\epsilon\right)\left(T_k\left(u_\epsilon\right)-T_k(u)\right)\right] d x .
  \end{gathered}
  \end{equation}
  
  In addition, we see that $h\left(u_\epsilon\right)\left(T_k\left(u_\epsilon\right)-T_k(u)\right) \rightharpoonup 0$ weakly in $W_0^{1, {\overrightarrow{p}}}(\Omega)$. Indeed the sequence  $\left(h\left(u_\epsilon\right)\left(T_k\left(u_\epsilon\right)-T_k(u)\right)\right)_{\epsilon>0}$ is bounded in $W_0^{1, {\overrightarrow{p}}}(\Omega)$ and converges to zero almost everywhere on $\Omega$. This implies that
  $$
  \lim _{\epsilon \rightarrow 0} \int_{\Omega} F \cdot D\left[h\left(u_\epsilon\right)\left(T_k\left(u_\epsilon\right)-T_k(u)\right)\right] d x=0 .
  $$
  Note also that, by generalized dominated convergence theorem, we have
  $$
  \lim _{\epsilon \rightarrow 0} \int_{\Omega} f_\epsilon h\left(u_\epsilon\right)\left(T_k\left(u_\epsilon\right)-T_k(u)\right) d x=0 .
  $$
  Let us now prove that
  \begin{equation}\label{3.12}
  \limsup _{\epsilon \rightarrow 0} \int_{\Omega} \beta_\epsilon\left(u_\epsilon\right) h\left(u_\epsilon\right)\left(T_k\left(u_\epsilon\right)-T_k(u)\right) d x \geq 0 .
  \end{equation}
  For any $0<r$ sufficiently small we consider
  $$
  u_r=\left(T_l(u) \wedge(M-r)\right) \vee(m+r),
  $$
  where $l$ is such that $\operatorname{supp}(h) \subset[-l,+l]$. For any $k>0, T_k\left(u_r\right) \in W_0^{1, {\overrightarrow{p}}}(\Omega)$. Furthermore, since
  $$
  \int_{\Omega} h\left(u_\epsilon\right)\left(\beta_\epsilon\left(u_\epsilon\right)-\beta_\epsilon\left(u_{\Gamma}\right)\right)\left(T_k\left(u_\epsilon\right)-T_k\left(u_{\Gamma}\right)\right) d x \geq 0,
  $$
  we have
  $$
  \begin{aligned}
  \int_{\Omega} \beta_\epsilon\left(u_\epsilon\right) h\left(u_\epsilon\right)\left(T_k\left(u_\epsilon\right)-T_k(u)\right) d x \geq & \int_{\Omega} h\left(u_\epsilon\right) \beta_\epsilon\left(u_r\right)\left(T_k\left(u_\epsilon\right)-T_k\left(u_{r}\right)\right) d x \\
  & +\int_{\Omega} h\left(u_\epsilon\right) \beta_\epsilon\left(u_\epsilon\right)\left(T_k\left(u_T\right)-T_k(u)\right) d x .
  \end{aligned}
  $$
  See that
  $$
  \max (m+r,-l) \leq u_r \leq \min (M-r, l),
  $$
  so that
  $$
  \beta_\epsilon(\max (m+r,-l)) \leq \beta_\epsilon\left(u_r\right) \leq \beta_\epsilon(\min (M-r, l)) .
  $$

  Using Lebesgue dominated convergence theorem, we get that
  $$
  \limsup _{\epsilon \rightarrow 0} \int_{\Omega} h\left(u_\epsilon\right) \beta_\epsilon\left(u_r\right)\left(T_k\left(u_\epsilon\right)-T_k\left(u_r\right)\right) d x=\int_{\Omega} h(u) \beta_0\left(u_r\right)\left(T_k(u)-T_k\left(u_r\right)\right) d x .
  $$
  As to the last term
  $$
  I:=\int_{\Omega} h\left(u_\epsilon\right) \beta_\epsilon\left(u_\epsilon\right)\left(T_k\left(u_r\right)-T_k(u)\right) d x
  $$
  see that
  
  $$
  \begin{aligned}
  I= & \int_{\Omega} f_\epsilon h\left(u_\epsilon\right)\left(T_k\left(u_{\Gamma}\right)-T_k(u)\right) d x+\int_{\Omega} F \cdot D\left[h\left(u_\epsilon\right)\left(T_k\left(u_{\Gamma}\right)-T_k(u)\right)\right] d x \\
  & -\int_{\Omega} a\left(x, D u_\epsilon\right) \cdot D\left[h\left(u_\epsilon\right)\left(T_k\left(u_r\right)-T_k(u)\right)\right] d x \\
  = & \int_{\Omega} f_\epsilon h\left(u_\epsilon\right)\left(T_k\left(u_r\right)-T_k(u)\right) d x+\int_{\Omega} F \cdot D\left[h\left(u_\epsilon\right)\left(T_k\left(u_{r}\right)-T_k(u)\right)\right] d x \\
  & -\int_{\Omega} h\left(u_\epsilon\right) a\left(x, D u_\epsilon\right) \cdot D\left(T_k\left(u_{r}\right)-T_k(u)\right) dx\\
    &-\int_{\Omega} h^{\prime}\left(u_{\varepsilon}\right)\left(T_k\left(u_{r}\right)-T_k(u)\right) a\left(x, D u_\epsilon\right)\cdot D u_\epsilon d x .
  \end{aligned}
  $$
  
  We need to let first $\varepsilon \rightarrow 0$ and next $r \rightarrow 0$. The three first terms are obvious. As to the last term, see that
  $$
  \begin{aligned}
  \left|T_k\left(u_r\right)-T_k(u)\right| \leq \left|\left(T_k(u)-T_k(M-r)\right) \chi_{[M-r \leq u \leq M]}\right|\\
     +\left|\left(T_k(u)-T_k(m+r)\right) \chi_{[m \leq u \leq m+r]}\right| \leq 2 r,
  \end{aligned}
  $$
  which implies that
  $$
  \begin{aligned}
  & \lim _{r \rightarrow 0} \limsup _{\epsilon \rightarrow 0}\left|\int_{\Omega} h^{\prime}\left(u_\epsilon\right)\left(T_k\left(u_r\right)-T_k(u)\right) a\left(x, D u_\epsilon\right) \cdot D u_\epsilon d x\right| \\
  & \quad=\lim _{\tau \rightarrow 0} \limsup _{\epsilon \rightarrow 0} 2 r\left|\int_{\Omega} h^{\prime}\left(u_\epsilon\right) a\left(x, D u_\epsilon\right) \cdot D u_\epsilon d x\right| \\
  & =0,
  \end{aligned}
  $$
  
  where we use the fact that $\left|\int_{\Omega} h^{\prime}\left(u_{\varepsilon}\right) a\left(x, D u_\epsilon\right) \cdot D u_\epsilon d x\right|$ is bounded.
  Now, see that
  $$
  h(u) \beta_0\left(u_r\right)\left(T_k(u)-T_k\left(u_r\right)\right) \geq 0 .
  $$
  Indeed.
  $$
  \begin{aligned}
  h(u) \beta_0\left(u_r\right)\left(T_k(u)-T_k\left(u_r\right)\right)= & h(u) \beta_0(M-r)\left(T_k(u)-T_k(M-r)\right) \chi[M-r \leq u \leq M] \\
  & +h\left(u_\epsilon\right) \beta_0(m+r)\left(T_k(u)-T_k(m+r)\right) \chi_{[m \leq u \leq m+r]} \geq 0,
  \end{aligned}
  $$

  where we use the fact that $\beta_0(M-r) \geq 0$ and $\beta_0(m+r) \leq 0$ (since $0 \in \beta(0)$ and $\left.m+r \leq 0 \leq M-r\right)$. This gives (\ref{3.12}).
  
  Passing to the limit in (\ref{3.11}) and using the above results we obtain the inequality (\ref{3.10}).
  
  Step 2: We prove that
  \begin{equation}\label{3.13}
  \limsup _{l \rightarrow+\infty} \limsup _{\epsilon \rightarrow 0} \int_{\left[l<| u_\epsilon \mid<l+1\right]} a\left(x, D u_\epsilon\right) \cdot D u_\epsilon d x \leq 0 .
  \end{equation}
  
  Taking $w_l\left(u_\epsilon\right)$ as test function in (3.2), where $w_l(r)=T_1\left(r-T_l(r)\right)$, we get
  $$
  \begin{gathered}
  \int_{\Omega} a\left(x, D u_\epsilon\right) \cdot D T_1\left(u_\epsilon-T_l\left(u_\epsilon\right)\right) d x+\int_{\Omega} \beta_\epsilon\left(u_\epsilon\right) T_1\left(u_\epsilon-T_l\left(u_\epsilon\right)\right) d x \\
  =\int_{\Omega} f_\epsilon T_1\left(u_\epsilon-T_l\left(u_\epsilon\right)\right) d x+\int_{\Omega} F \cdot D T_1\left(u_\epsilon-T_l\left(u_\epsilon\right)\right) d x .
  \end{gathered}
  $$
  Since
  $$
  \int_{\Omega} \beta_\epsilon\left(u_\epsilon\right) T_1\left(u_\epsilon-T_l\left(u_\epsilon\right)\right) d x \geq 0
  $$
  
  and
  
  $$
  \int_{\Omega} a\left(x, D u_\epsilon\right) \cdot D T_1\left(u_\epsilon-T_l\left(u_\epsilon\right)\right) d x=\int_{\left[ l <\left|u_\epsilon\right| <l+1\right]} a\left(x, D u_\epsilon\right) \cdot D u_\epsilon d x
  $$
  
  we get
  \begin{equation}\label{3.14}
  \int_{\left[ l <\left|u_\epsilon\right| <l+1\right]} a\left(x, D u_\epsilon\right) \cdot D u_\epsilon d x \leq \int_{\Omega} f_\epsilon T_1\left(u_\epsilon-T_I\left(u_\epsilon\right)\right) d x+\int_{\Omega} F \cdot D T_1\left(u_\epsilon-T_I\left(u_\epsilon\right)\right) d x .
  \end{equation}
  
  Recall that by Step 1, we have
  $$
  \lim _{\epsilon \rightarrow 0} \int_{\Omega} f_\epsilon T_1\left(u_\epsilon-T_l\left(u_\epsilon\right)\right) d x=\int_{\Omega} f T_1\left(u-T_I(u)\right) d x .
  $$
  So, using the fact that $T_1\left(u-T_l(u)\right) \longrightarrow 0$ a.e. in $\Omega$ as $l \longrightarrow+\infty$ and the Lebesgue dominated convergence theorem, we obtain
  $$
  \lim _{l \rightarrow+\infty} \lim _{\epsilon \rightarrow 0} \int_{\Omega} f_\epsilon T_1\left(u_\epsilon-T_l\left(u_\epsilon\right)\right) d x=0 .
  $$
  Now, let us see that $\lim _{l \rightarrow+\infty} \lim _{\epsilon \rightarrow 0} \int_{\Omega} F \cdot D T_1\left(u_\epsilon-T_l\left(u_\epsilon\right)\right) d x=0$. Indeed we begin by proving that
  $$
  \lim _{l \rightarrow+\infty} \lim _{\epsilon \rightarrow 0} \sum_{i=1}^{N}\int_{\left[|<| u_\epsilon < l+1\right]}\left|D u_\epsilon\right|^{p_i} d x=0
  $$
  
  Thanks to $(\mathbf{H}_{1})$, we have
  $$
  \begin{aligned}
  & \lambda \sum_{i=1}^{N}\int_{[l<\left|u_\epsilon\right|<l+1 ]}\left|D u_\epsilon\right|^{p_i} d x \leq \int_{\left[ l <\left|u_\epsilon\right| <l+1\right]} a\left(x, D u_\epsilon\right) \cdot D u_\epsilon d x \\
  & \leq \int_{\Omega} f_\epsilon T_1\left(u_\epsilon-T_I\left(u_\epsilon\right)\right) d x+\int_{\left[ l <\left|u_\epsilon\right| <l+1\right]} F \cdot D u_\epsilon d x . \\
  &
  \end{aligned}
  $$
  Using Young's inequality. for every $\tilde{\epsilon}>0$, we get
  
  $$
  \int_{\left[ l <\left|u_\epsilon\right| <l+1\right]} F \cdot D u_\epsilon d x \leq \sum_{i=1}^{N}\frac{(\tilde{\epsilon})^{1-p_i^{\prime}}}{p_i^{\prime}} \int_{\left[ l <\left|u_\epsilon\right| <l+1\right]}|F|^{p_i^{\prime}} d x+ \sum_{i=1}^{N}\frac{\tilde{\epsilon}}{p_i} \int_{\left[ l <\left|u_\epsilon\right| <l+1\right]}\left|D u_\epsilon\right|^{p_i} d x .
  $$
  
  Taking $\tilde{\epsilon}=\frac{p_i}{2} \lambda$ we obtain
  $$
  \frac{\lambda}{2}\sum_{i=1}^{N} \int_{\left[ l <\left|u_\epsilon\right| <l+1\right]}\left|D u_\epsilon\right|^{p_i} d x \leq \int_{\Omega} f_\epsilon T_1\left(u_\epsilon-T_I\left(u_\epsilon\right)\right) d x+C(\lambda, p_i) \sum_{i=1}^{N}\int_{\left[ l <\left|u_\epsilon\right| <l+1\right]}|F|^{p_i^{\prime}} d x .
  $$
  Furthermore
  $$
  \sum_{i=1}^{N}\int_{\left[ l <\left|u_\epsilon\right| <l+1\right]}|F|^{p_i^{\prime}} d x \leq \sum_{i=1}^{N}\int_{[| u_\epsilon| >l]}|F|^{p_i^{\prime}} d x
  $$
  and
  $$
  \lim _{\epsilon \rightarrow 0}\sum_{i=1}^{N} \int_{[| u_\epsilon| >l]}|F|^{p_i^{\prime}} d x \leq \sum_{i=1}^{N}\int_{[| u | \geq 1]}|F|^{p_i^{\prime}} d x
  $$
  Since
  $$
  \operatorname{meas}([|u| \geq l]) \longrightarrow 0, \quad \text { as } l \longrightarrow+\infty \text { by }(\ref{2.2})
  $$
  we have
  $$
  \lim _{l \rightarrow+\infty} \lim _{\epsilon \rightarrow 0}\sum_{i=1}^{N} \int_{\left[ l <\left|u_\epsilon\right|<l+1\right]}|F|^{p_i^{\prime}} d x=0 .
  $$
  Hence, $\lim _{l \rightarrow+\infty} \lim _{\epsilon \rightarrow 0} \sum_{i=1}^{N}\int_{\left[ l <\left|u_\epsilon\right| <l+1\right]}\left|D u_\epsilon\right|^{{p_i}} d x=0$. Now, using the above results we obtain
  $$
  \lim _{l \rightarrow+\infty} \lim _{\epsilon \rightarrow 0} \int_{\Omega} F \cdot D T_1\left(u_\epsilon-T_l\left(u_\epsilon\right)\right) d x=0 .
  $$
  Then passing to the limit as $\epsilon \longrightarrow 0$ and as $l \longrightarrow+\infty$ in (\ref{3.14}) we get (\ref{3.13}).

  Step 3: We prove that for every $k>0$,
  \begin{equation}\label{3.15}
  \limsup _{\epsilon \rightarrow 0} \int_{\Omega} a\left(x, D T_k\left(u_\epsilon\right)\right) \cdot\left[D T_k\left(u_\epsilon\right)-D T_k(u)\right] d x \leq 0.
  \end{equation}
  
  Indeed, for all $l>0$ we define the function $h_l$ by $h_l(r)=\inf \left[1,(l+1-|r|)^{+}\right\}$. For $l>k$, we have
  
  $$
  \begin{aligned}
  \int_{\Omega} a\left(x, D u_\epsilon\right) \cdot D\left[h_l\left(u_\epsilon\right)\left(T_k\left(u_\epsilon\right)-T_k(u)\right)\right] d x=&
   \int_{\left[\left|u_\epsilon\right| \leq k\right]} h_l\left(u_\epsilon\right) a\left(x, D T_k\left(u_\epsilon\right)\right).[D T_k\left(u_\epsilon\right)-DT_k(u)]d x\\
  & +\int_{\left[\left|u_\epsilon\right|>k\right]} h_l\left(u_\epsilon\right) a\left(x, D u_\epsilon\right)\left(-D T_k(u)\right) d x \\
  & +\int_{\Omega} h_l^{\prime}\left(u_\epsilon\right)\left(T_k\left(u_\epsilon\right)-T_k(u)\right) a\left(x, D u_\epsilon\right). D u_\epsilon d x \\
  := & \left(E_1\right)+\left(E_2\right)+\left(E_3\right) .
  \end{aligned}
  $$
  Since $l>k$, on the set $\left[\left|u_\epsilon\right| \leq k\right]$ we have $h_l\left(u_\epsilon\right)=1$ so that we can write $\left(E_1\right)$ as
  $$
  \begin{aligned}
  \left(E_1\right) & =\int_{\left[\left|u_\epsilon\right| \leq k\right]} a\left(x, D T_k\left(u_\epsilon\right)\right) \cdot\left[D T_k\left(u_\epsilon\right)-D T_k(u)\right] d x \\
  & =\int_{\Omega} a\left(x, D T_k\left(u_\epsilon\right)\right) \cdot\left[D T_k\left(u_\epsilon\right)-D T_k(u)\right] d x .
  \end{aligned}
  $$
  Hence we obtain
  $$
  \limsup _{\epsilon \rightarrow 0}\left(E_1\right)=\limsup _{\epsilon \rightarrow 0} \int_{\Omega} a\left(x, D T_k\left(u_\epsilon\right)\right) \cdot\left[D T_k\left(u_\epsilon\right)-D T_k(u)\right] d x .
  $$
  Let us write the term $\left(E_2\right)$ as
  $$
  \left(E_2\right)=-\int_{\left[\left|u_\epsilon\right| > k\right]} h_l\left(u_\epsilon\right) c\left(x, D T_{l+1}\left(u_\epsilon\right)\right) \cdot D T_k(u) d x .
  $$
  Using Lebesgue dominated convergence theorem, we get
  $$
  \lim _{\epsilon \rightarrow 0}\left(E_2\right)=-\int_{\left[\left|u_\epsilon\right| \geq k\right]} h_l(u) H_{l+1} \cdot D T_k(u) d x=0 .
  $$
  For the term $\left(E_3\right)$, we have
  $$
  \begin{aligned}
  \left(-\int_{\Omega} h_l^{\prime}\left(u_\epsilon\right)\left(T_k\left(u_\epsilon\right)-T_k(u)\right) a\left(x, D u_\epsilon\right). Du_\epsilon d x\right) & \leq\left|\int_{\Omega} h_l^{\prime}\left(u_\epsilon\right)\left(T_k\left(u_\epsilon\right)-T_k(u)\right) a\left(x, D u_\epsilon\right) \cdot D u_\epsilon d x\right| \\
  & \leq 2 k \int_{\left[ l <\left|u_\epsilon\right| <l+1\right]} a\left(x, D u_\epsilon\right) \cdot D u_\epsilon d x .
  \end{aligned}
  $$

  Using the result of Step 2 we deduce that
  $$
  \limsup _{l \rightarrow+\infty} \limsup _{\epsilon \rightarrow 0}\left(-\int_{\Omega} h_l^{\prime}\left(u_\epsilon\right)\left(T_k\left(u_\epsilon\right)-T_k(u)\right) a\left(x, D u_\epsilon\right) . D u_\epsilon d x\right) \leq 0 .
  $$
  Applying (\ref{3.10}) with $h$ replaced by $h_l, l>k$ we get
  $$
  \begin{aligned}
  & \limsup _{\epsilon \rightarrow 0} \int_{\Omega} a\left(x, D T_k\left(u_\epsilon\right)\right) \cdot\left[D\left(T_k\left(u_\epsilon\right)- DT_k(u)\right)\right] d x \\
  & \leq \limsup _{\epsilon \rightarrow 0}\left(-\int_{\Omega} h_l^{\prime}\left(u_\epsilon\right)\left(T_k\left(u_\epsilon\right)-T_k(u)\right) a\left(x, D u_\epsilon\right) \cdot D u_\epsilon d x\right) \text {, } \\
  &
  \end{aligned}
  $$
  so that letting $l \longrightarrow+\infty$ yields the inequality (\ref{3.15}).
  
  Step 4: In this step we prove by standard monotonicity arguments that for all $k>0, H_k=$ $a\left(x, D T_k(u)\right)$ a.e. in $\Omega$. Let $\varphi \in \mathcal{D}(\Omega)$ and $\tilde{\alpha} \in \mathbb{R}^*$. Using (\ref{3.15}), we have
  $$
  \begin{aligned}
  \tilde{\alpha} \lim _{\epsilon \rightarrow 0} \int_{\Omega} a\left(x, D T_k\left(u_\epsilon\right)\right) \cdot D \varphi d x & \geq \limsup _{\epsilon \rightarrow 0} \int_{\Omega} a\left(x, D T_k\left(u_\epsilon\right)\right) \cdot\left[D T_k\left(u_\epsilon\right)-D T_k(u)+D(\tilde{\alpha} \varphi)\right] d x \\
  & \geq \underset{\epsilon \rightarrow 0}{\limsup } \int_{\Omega} a\left(x, D\left(T_k(u)-\tilde{\alpha} \varphi\right)\right) \cdot\left[D T_k\left(u_\epsilon\right)-D T_k(u)+D(\tilde{\alpha} \varphi)\right] d x \\
  & \geq \limsup _{\epsilon \rightarrow 0} \int_{\Omega} a\left(x, D\left(T_k(u)-\tilde{\alpha} \varphi\right)\right) \cdot D(\tilde{\alpha} \varphi) d x \\
  & \geq \tilde{\alpha} \int_{\Omega} a\left(x, D\left(T_k(u)-\tilde{\alpha} \varphi\right)\right) \cdot D \varphi d x .
  \end{aligned}
  $$
  Dividing by $\tilde{\alpha}>0$ and by $\tilde{\alpha}<0$, passing the limit with $\tilde{\alpha} \longrightarrow 0$ we obtain
  $$
  \lim _{\epsilon \rightarrow 0} \int_{\Omega} a\left(x, D T_k\left(u_\epsilon\right)\right) \cdot D \varphi d x=\int_{\Omega} a\left(x, D T_k(u)\right) \cdot D \varphi d x .
  $$
  This means that $\forall k>0 . \int_{\Omega} H_k \cdot D \varphi d x=\int_{\Omega} a\left(x, D T_k(u)\right) . D \varphi d x$ and then
  $$
  H_k=a\left(x, D T_k(u)\right) \text { in } \mathcal{D}^{\prime}(\Omega)
  $$
  for all $k>0$. Hence $H_k=a\left(x, D T_k(u)\right)$ a.e. in $\Omega$ and so $a\left(x, D T_k\left(u_\epsilon\right)\right) \rightharpoonup a\left(x, D T_k(u)\right)$ weakly in $\prod_{i=1}^{N}L^{p_{i}^{^{%
  \prime }}}(\Omega )$.
  
  (ii) Thanks to (\ref{3.15}), we deduce that for all $k>0$ :
  $$
  \lim _{\epsilon \rightarrow 0} \int_{\Omega}\left[a\left(x, D T_k\left(u_\epsilon\right)\right)-a\left(x, D T_k(u)\right)\right] \cdot\left[D T_k\left(u_\epsilon\right)-D T_k(u)\right] d x=0 .
  $$
  Since
  $$
  g_\epsilon(.):=\left[a\left(., D T_k\left(u_\epsilon\right)\right)-a\left(., D T_k(u)\right)\right] \cdot\left[D T_k\left(u_\epsilon\right)-D T_k(u)\right] \geq 0,
  $$
  
  up to a subsequence we have
  $$
  g_\epsilon(.) \longrightarrow 0 \text { a.e. in } \Omega \text {. }
  $$
  This implies that, there exists $Z \subset \Omega$ such that meas $(Z)=0$ and $g_\epsilon(.) \longrightarrow 0$ a.e. in $\Omega \backslash Z$. Let $x \in$ $\Omega \backslash Z$. Using the assumptions $(\mathbf{H}_{1})$ and $(\mathbf{H}_{2})$, it follows that the sequence $\left(D T_k\left(u_\epsilon(x)\right)\right)_{\epsilon>0}$ is bounded in $\mathbb{R}^N$ so we can extract a subsequence which converges to some $\tilde{\xi}$ in $\mathbb{R}^N$. Passing to the limit in the expression of $g_\epsilon(x)$, we get
  $$
  0=\left[a(x, \tilde{\xi})-a\left(x, D T_k(u(x))\right)\right] \cdot\left[\tilde{\xi}-D T_k(u(x))\right] .
  $$
  This yields $\tilde{\xi}=D T_k(u(x)), \forall x \in \Omega \backslash Z$. As the limit does not depend on the subsequence, the whole sequence $\left(D T_k\left(u_\epsilon(x)\right)\right)_{\epsilon>0}$ converges to $\tilde{\xi}$ in $\mathbb{R}^N$. This means that
  $$
  D T_k\left(u_\epsilon\right) \longrightarrow D T_k(u) \text { a.e. in } \Omega \text {. }
  $$
  (iii) The continuity of $a(x, \xi)$ with respect to $\xi \in \mathbb{R}^N$ gives us
  $$
  a\left(x, D T_k\left(u_\epsilon\right)\right) \longrightarrow a\left(x, D T_k(u)\right) \text { a.e. in } \Omega
  $$
  and then we obtain
  $$
  a\left(x, D T_k\left(u_\epsilon\right)\right) \cdot D T_k\left(u_\epsilon\right) \longrightarrow a\left(x, D T_k(u)\right) \cdot D T_k(u) \text { a.e. in } \Omega .
  $$
  Setting $y_\epsilon=a\left(x, D T_k\left(u_\epsilon\right)\right) . D T_k\left(u_\epsilon\right)$ and $y=a\left(x, D T_k(u)\right) . D T_k(u)$, we have
  $$
  \left\{\begin{array}{l}
  y_\epsilon \geq 0, \quad y_\epsilon \longrightarrow y \text { a.e. in } \Omega, y \in L^1(\Omega), \\
  \int_{\Omega} y_\epsilon d x \longrightarrow \int_{\Omega} y d x .
  \end{array}\right.
  $$
  Since
  $$
  \int_{\Omega}\left|y_\epsilon-y\right| d x=2 \int_{\Omega}\left(y-y_\epsilon\right)^{+} d x+\int_{\Omega}\left(y_\epsilon-y\right) d x
  $$
  and $\left(y-y_\epsilon\right)^{+} \leq y$ it follows by the Lebesgue dominated convergence theorem that
  $$
  \lim _{\epsilon \rightarrow 0} \int_{\Omega}\left|y_\epsilon-y\right| d x=0,
  $$
  which means that
  $$
  a\left(x, D T_k\left(u_\epsilon\right)\right) \cdot D T_k\left(u_\epsilon\right) \longrightarrow a\left(x, D T_k(u)\right) \cdot D T_k(u) \text { strongly in } L^1(\Omega) .
  $$
  (iv) By $(\mathbf{H}_{1})$, we have $\sum_{i=1}^{N}\left|D T_k\left(u_\epsilon\right)\right|^{p_i} \leq \frac{1}{\lambda}\sum_{i=1}^{N} a_i\left(x, D T_k\left(u_\epsilon\right)\right) . D T_k\left(u_\epsilon\right)$. Using the $L^1$-convergence of (iii) and the generalized dominated convergence theorem, we obtain the result of (iv).
  \end{proof}

  \begin{lemma}\label{L3.1}
   For any $h \in C_c^1(\mathbb{R})$ and $\xi \in W_0^{1, {\overrightarrow{p}}}(\Omega) \cap L^{\infty}(\Omega)$.
  $$
  D\left[h\left(u_\epsilon\right) \xi\right] \longrightarrow D[h(u) \xi] \text { strongly in } L^{\overrightarrow{p}}(\Omega) \text { as } \epsilon \longrightarrow 0 .
  $$
  \end{lemma}
  
  \begin{proof}

   For any $h \in C_c^1(\mathbb{R})$ and $\xi \in W_0^{1, {\overrightarrow{p}}}(\Omega) \cap L^{\infty}(\Omega)$. we have
  $$
  \begin{aligned}
  D\left[h\left(u_\epsilon\right) \xi\right] & =h\left(u_\epsilon\right) D \xi+h^{\prime}\left(u_\epsilon\right) \xi D u_\epsilon \\
  & \left.=h\left(u_\epsilon\right) D \xi+h^{\prime}\left(u_\epsilon\right) \xi D T_l\left(u_\epsilon\right) \text { for } l>0 \text { such that } \operatorname{supp}(h) \subset\right]-l,+l[
  \end{aligned}
  $$
  Using the Lebesgue dominated convergence theorem, we get
  $$
  h\left(u_\epsilon\right) D \xi \longrightarrow h(u) D \xi \quad \text { strongly in } L^{\overrightarrow{p}}(\Omega) \text { as } \epsilon \longrightarrow 0 \text {. }
  $$
  Moreover, since $\left|h^{\prime}\left(u_\epsilon\right) \xi D T_l\left(u_\epsilon\right)\right| \leq C\left|D T_l\left(u_\epsilon\right)\right|$. then using the generalized convergence theorem and Lemma \ref{P3.2} (iv) we deduce that
  $$
  h^{\prime}\left(u_\epsilon\right) \xi D T_I\left(u_\epsilon\right) \longrightarrow h^{\prime}(u) \xi D T_l(u)=h^{\prime}(u) \xi D u \quad \text { strongly in } L^{\overrightarrow{p}}(\Omega) \text { as } \epsilon \longrightarrow 0 .
  $$
  So Lemma \ref{L3.1} follows.
  \end{proof}
  
  Now, to pass to the limit in $\beta_{\varepsilon}\left(u_{\varepsilon}\right)$. let us consider the function $h_0=h_{l_0}, l_0>0$ to be fixed later such that
  $$
  \left\{\begin{array}{l}
  h_0 \in C_c^1(\mathbb{R}), \quad h_0(r) \geq 0, \quad \forall r \in \mathbb{R}, \\
  h_0(r)=1 \text { if }|r| \leq l_0 \text { and } h_0(r)=0 \text { if }|r| \geq l_0+1 .
  \end{array}\right.
  $$
  Since, for any $k>0$, $\left(h_k\left(u_\epsilon\right) \beta_\epsilon\left(u_\epsilon\right)\right)_{\epsilon>0}$ is bounded in $L^1(\Omega)$, there exists $z_k \in \mathcal{M}_b(\Omega)$, such that
  $$
  h_k\left(u_\epsilon\right) \beta_\epsilon\left(u_\epsilon\right) \stackrel{*}{\rightharpoonup} z_k \quad \text { in } \mathcal{M}_b(\Omega) \text { as } \epsilon \longrightarrow 0 .
  $$
  Moreover, for any $\xi \in W_0^{1, {\overrightarrow{p}}}(\Omega) \cap L^{\infty}(\Omega)$. we have
  $$
  \int_{\Omega} \xi d z_k=\int_{\Omega} \xi h_k(u) d \mu-\int_{\Omega} a(x, D u) \cdot D\left(h_k(u) \xi\right) d x,
  $$
  which implies that $z_k \in \mathcal{M}_b^{\overrightarrow{p}}(\Omega)$ and, for any $k \leq l$,
  $$
  z_{\mathrm{k}}=z_{\mathrm{l}} \text { on }\left[\left|T_k(u)\right|<k\right] \text {. }
  $$
  Let us consider the Radon measure $z$ defined by
  \begin{equation}\label{3.16}
  \begin{cases}z=z_k, & \text { on }\left[\left|T_k(u)\right|<k\right] \text { for } k \in \mathbb{N}^*, \\ z=0 & \text { on } \bigcap_{k \in \mathbb{N}^*}\left[\left|T_k(u)\right|=k\right] .
  \end{cases}
  \end{equation}

  For any $h \in \mathcal{C}_c(\mathbb{R}), h(u) \in L^{\infty}(\Omega, d|z|)$ and
  $$
  \int_{\Omega} h(u) \xi d z=-\int_{\Omega} a(x, D u) \cdot D(h(u) \xi) d x+\int_{\Omega} h(u) \xi d \mu,
  $$
  for any $\xi \in W_0^{1, {\overrightarrow{p}}}(\Omega) \cap L^{\infty}(\Omega)$. Indeed, let $k_0>0$ be such that $\operatorname{supp}(h) \subseteq\left[-k_0,+k_0\right]$.
  $$
  \begin{aligned}
  \int_{\Omega} h(u) \xi d z & =\int_{\Omega} h(u) \xi d z_{k_0} \\
  & =-\lim _{\varepsilon \rightarrow 0} \int_{\Omega} a\left(x, D u_{\varepsilon}\right) \cdot D\left(h\left(u_{\varepsilon}\right) \xi\right)+\lim _{\varepsilon \rightarrow 0} \int_{\Omega} h\left(u_{\varepsilon}\right) \xi d \mu_{\varepsilon} \\
  & =-\lim _{\varepsilon \rightarrow 0} \int_{\Omega} a\left(x, D T_{k_0}\left(u_{\varepsilon}\right)\right) \cdot D\left(h\left(u_{\varepsilon}\right) \xi\right)+\lim _{\varepsilon \rightarrow 0} \int_{\Omega} h\left(u_{\varepsilon}\right) \xi d \mu_{\varepsilon} \\
  & =-\int_{\Omega} a(x, D u) \cdot D(h(u) \xi)+\int_{\Omega} h(u) \xi d \mu .
  \end{aligned}
  $$
  Moreover, we have
  \begin{lemma}\label{L3.2}

   The Radon-Nikodym decomposition of the measure $z$ given by (\ref{3.16}) with respect to $\mathcal{L}^N$.
  $$
  z=w \mathcal{L}^N+v, \quad \text { with } v \perp \mathcal{L}^N
  $$
  satisfies the following properties
  $$
  \left\{\begin{array}{l}
  w \in \beta(u) \quad \mathcal{L}^N \text {-a.e. in } \Omega, \quad w \in L^1(\Omega), \quad v \in \mathcal{M}_b^{\overrightarrow{p}}(\Omega), \\
  v^{+} \text {is concentrated on }[u=M] \cap[u \neq+\infty] \text { and } \\
  v^{-} \text {is concentrated on }[u=m] \cap[u \neq-\infty] .
  \end{array}\right.
  $$
  
  \end{lemma}
  
  \begin{proof}

  Since, for any $\epsilon>0, z_\epsilon \in \partial j_\epsilon\left(u_\epsilon\right)$. we have
  $$
  j(t) \geq j_\epsilon(t) \geq j_\epsilon\left(u_\epsilon\right)+\left(t-u_\epsilon\right) z_\epsilon \quad \mathcal{L}^N \text {-a.e. in } \Omega, \forall t \in \mathbb{R} .
  $$
  Then for any $h \in \mathcal{C}_c(\mathbb{R}), h \geq 0, \xi \geq 0$ and $k>0$ such that supp $(h) \subseteq[-k, k]$, we have
  $$
  \xi h\left(u_\epsilon\right) j(t) \geq \xi h\left(u_\epsilon\right) j_\epsilon\left(u_\epsilon\right)+\left(t-u_\epsilon\right) \xi h\left(u_\epsilon\right) h_k\left(u_{\varepsilon}\right) z_\epsilon .
  $$
  In addition, for any $0<\epsilon<\tilde{\epsilon}$, we have
  $$
  \xi h\left(u_\epsilon\right) j(t) \geq \xi h\left(u_\epsilon\right) j_{\tilde{\epsilon}}\left(u_\epsilon\right)+\left(t-u_\epsilon\right) \xi h\left(u_\epsilon\right) h_k\left(u_{\varepsilon}\right) z_\epsilon
  $$
  and, integrating over $\Omega$ yields
  $$
  \int_{\Omega} \xi h\left(u_\epsilon\right) j(t) d x \geq \int_{\Omega} \xi h\left(u_\epsilon\right) j_{\tilde{\epsilon}}\left(u_\epsilon\right) d x+\int_{\Omega}\left(t-u_\epsilon\right) \xi h\left(u_\epsilon\right) h_k\left(u_{\varepsilon}\right) z_\epsilon d x .
  $$

  As $\in \longrightarrow 0$, we get by using Fatou's Lemma
  $$
  \int_{\Omega} \xi h(u) j(t) d x \geq \int_{\Omega} \xi h(u) j_\epsilon(u) d x+\liminf _{\epsilon \rightarrow 0} \int_{\Omega}\left(t-u_\epsilon\right) \xi h\left(u_\epsilon\right) h_k\left(u_{\varepsilon}\right) z_\epsilon d x .
  $$
  Now, for any $\xi \in C_c^1(\Omega)$ and $t \in \mathbb{R}$, setting
  $$
  \tilde{h}(r)=(t-r) h(r),
  $$
  we have
  $$
  \begin{aligned}
  \lim _{\epsilon \rightarrow 0} \int_{\Omega}\left(t-u_\epsilon\right) h\left(u_\epsilon\right) \xi h_k\left(u_{\varepsilon}\right) z_\epsilon d x & =\lim _{\epsilon \rightarrow 0} \int_{\Omega} \tilde{h}\left(u_\epsilon\right) \xi h_k\left(u_{\varepsilon}\right) z_\epsilon d x \\
  & =\int_{\Omega}(t-u) h(u) \xi d z_k \\
  & =\int_{\Omega}(t-u) h(u) \xi d z .
  \end{aligned}
  $$
  So,
  $$
  \int_{\Omega} \xi h(u) j(t) d x \geq \int_{\Omega} \xi h(u) j_\epsilon(u) d x+\int_{\Omega} \xi(t-u) h(u) d z .
  $$
  As $\tilde{\epsilon} \longrightarrow 0$, we get by using again Fatou's Lemma
  $$
  \int_{\Omega} \xi h(u) j(t) d x \geq \int_{\Omega} \xi h(u) j(u) d x+\int_{\Omega} \xi(t-u) h(u) d z .
  $$
  From the inequality above we have
  
  \begin{equation}\label{3.17}
  h(u) j(t) \geq h(u) j(u)+(t-u) h(u) z, \quad \text { in } \mathcal{M}_b(\Omega), \forall t \in \mathbb{R} .
  \end{equation}
  
  Using the Radon-Nikodym decomposition of $z$ we have $z=w \mathcal{L}^N+v$ with $v \perp \mathcal{L}^N, w \in L^1(\Omega)$, then comparing the regular part and the singular part of (\ref{3.17}), for any $h \in \mathcal{C}_c(\mathbb{R})$. we obtain
  \begin{equation}\label{3.18}
  h(u) j(t) \geq h(u) j(u)+(t-u) h(u) w \quad \mathcal{L}^N \text {-a.e. in } \Omega, \forall t \in \mathbb{R}
  \end{equation}
  
  and
  
  \begin{equation}\label{3.19}
  (t-u) h(u) v \leq 0 \text { in } \mathcal{M}_b(\Omega), \forall t \in \overline{\operatorname{dom}(j)}.
  \end{equation}
  
  From (\ref{3.18}) we get
  $$
  j(t) \geq j(u)+(t-u) w \quad \mathcal{L}^N \text {-a.e. in } \Omega, \forall t \in \mathbb{R},
  $$
  so that $w \in \partial j(u) \mathcal{L}^N$-a.e. in $\Omega$. As to (\ref{3.19}), this implies that for any $t \in \overline{\operatorname{dom}(j)}$.

  \begin{equation}\label{3.20}
  v \geq 0 \text { in }[u \in(t, \infty) \cap \operatorname{supp}(h)]
  \end{equation}
  
  and

  \begin{equation}\label{3.21}
  v \leq 0 \text { in }[u \in(-\infty, t) \cap \operatorname{supp}(h)] .
  \end{equation}
  
  In particular, this implies that
  $$
  v([m<u<M])=0 .
  $$
  Moreover, if $m \neq-\infty$ (resp. $M \neq+\infty$ ), then (\ref{3.21}) (resp. (\ref{3.20})) implies that
  $$
  v^{-} \text {is concentrated on }[u=m] \quad \text { (resp. } v^{+} \text {is concentrated on }[u=M] \text { ). }
  $$
  By construction of $z$, we see that
  $$
  v([u= \pm \infty])=0,
  $$
  and the proof of Lemma \ref{L3.2} is finished.
  \end{proof}
  
  To finish the proof of Theorem \ref{1.1}. we consider $\xi \in W_0^{1, {\overrightarrow{p}}}(\Omega) \cap L^{\infty}(\Omega)$, and $h \in C_c^1(\mathbb{R})$. Then, we take $h\left(u_\epsilon\right) \xi$ as test function in (\ref{3.2}). We get
  \begin{equation}\label{3.22}
  \int_{\Omega} a\left(x, D u_\epsilon\right) \cdot D\left[h\left(u_\epsilon\right) \xi\right] d x+\int_{\Omega} \beta_\epsilon\left(u_\epsilon\right) h\left(u_\epsilon\right) \xi d x=\int_{\Omega} h\left(u_\epsilon\right) \xi f_\epsilon d x+\int_{\Omega} F \cdot D\left[h\left(u_\epsilon\right) \xi\right] d x .
  \end{equation}
  
  Using Lemma \ref{L3.1}, it is not difficult to see that
  $$
  \lim _{\epsilon \rightarrow 0}\left(\int_{\Omega} h\left(u_\epsilon\right) \xi f_\epsilon d x+\int_{\Omega} F \cdot D\left[h\left(u_\epsilon\right) \xi\right] d x\right)=\int_{\Omega} h(u) \xi d \mu.
  $$
  
  The first term of (\ref{3.22}) can be written as
  $$
  \int_{\Omega} a\left(x, D u_\epsilon\right) \cdot D\left[h\left(u_\epsilon\right) \xi\right] d x=\int_{\Omega} a\left(x, D T_{l_0+1}\left(u_\epsilon\right)\right) \cdot D\left[h_0\left(u_\epsilon\right) \xi\right] d x,
  $$
  for some $l_0>0$ so that, by Lemma \ref{P3.2} (i) and Lemma \ref{L3.1}, we have
  $$
  \begin{aligned}
  \lim _{\epsilon \rightarrow 0} \int_{\Omega} a\left(x, D u_\epsilon\right) \cdot D\left[h\left(u_\epsilon\right) \xi\right] d x & =\lim _{\epsilon \rightarrow 0} \int_{\Omega} a\left(x, D T_{l_0+1}\left(u_\epsilon\right)\right) \cdot D\left[h_0\left(u_\epsilon\right) \xi\right] d x \\
  & =\int_{\Omega} a\left(x, D T_{l_0+1}(u)\right) \cdot D\left[h_0(u) \xi\right] d x \\
  & =\int_{\Omega} a(x, D u) \cdot D[h(u) \xi] d x .
  \end{aligned}
  $$

  Thanks to the convergence of Lemma \ref{L3.2} we have from (\ref{3.22})
  $$
  \begin{aligned}
  \lim _{\epsilon \rightarrow 0} \int_{\Omega} \beta_\epsilon\left(u_\epsilon\right) h\left(u_\epsilon\right) \xi d x & =\int_{\Omega} h(u) \xi d \mu-\int_{\Omega} a(x, D u) \cdot D[h(u) \xi] d x \\
  & =\int_{\Omega} h(u) \xi d z \\
  & =\int_{\Omega} h(u) w \xi d x+\int_{\Omega} h(u) \xi d v .
  \end{aligned}
  $$
  
  Letting $\epsilon$ go to 0 in (\ref{3.22}) it yields that $(u, w)$ is a solution of the problem $(E,\mu)$.
  To end the proof of Theorem \ref{1.1}, we prove (\ref{1.9}). We take $\varphi=T_1\left(u_\epsilon-T_n\left(u_\epsilon\right)\right)$ as test function in (\ref{3.2}) to get
  \begin{equation}\label{3.23}
  \begin{gathered}
  \int_{\Omega} a\left(x, D u_\epsilon\right) \cdot D\left[T_1\left(u_\epsilon-T_n\left(u_\epsilon\right)\right)\right] d x+\int_{\Omega} \beta_\epsilon\left(u_\epsilon\right) T_1\left(u_\epsilon-T_n\left(u_\epsilon\right)\right) d x \\
  =\int_{\Omega} T_1\left(u_\epsilon-T_n\left(u_\epsilon\right)\right) f_\epsilon d x+\int_{\Omega} F \cdot D\left[T_1\left(u_\epsilon-T_n\left(u_\epsilon\right)\right)\right] d x .
  \end{gathered}
  \end{equation}

  Since $\int_{\Omega} \beta_\epsilon\left(u_\epsilon\right) T_1\left(u_\epsilon-T_n\left(u_\epsilon\right)\right) d x \geq 0$ and $D\left[T_1\left(u_\epsilon-T_n\left(u_\epsilon\right)\right)\right]=D u_\epsilon \chi_{[n<\left|u_\epsilon\right| < n+1]}.$  We have from equality (\ref{3.23})
  
  \begin{equation}\label{3.24}
  \int_{\left[n<\left|u_\epsilon\right|<n+1\right]} a\left(x, D u_\epsilon\right) \cdot D u_\epsilon d x \leq \int_{\Omega} T_1\left(u_\epsilon-T_n\left(u_\epsilon\right)\right) f_\epsilon d x+\int_{\Omega} F \cdot D\left[T_1\left(u_\epsilon-T_n\left(u_\epsilon\right)\right)\right] d x.
  \end{equation}
  
  As $\in \longrightarrow 0$ in (\ref{3.24}). we get
  \begin{equation}\label{3.25}
  \int_{[n \leq|u| \leq n+1]} a(x, D u) \cdot D u d x \leq \int_{\Omega} T_1\left(u-T_n(u)\right) f d x+\int_{\Omega} F \cdot D\left[T_1\left(u-T_n(u)\right)\right] d x .
  \end{equation}
  
  Using assumption $(\mathbf{H}_{1})$, it follows
  \begin{equation}\label{3.26}
  \lambda\sum_{i=1}^{N} \int_{[n \leq|u| \leq n+1]}|D u|^{p_i} d x \leq \int_{\Omega} T_1\left(u-T_n(u)\right) f d x+\int_{\Omega} F \cdot D\left[T_1\left(u-T_n(u)\right)\right] d x .
  \end{equation}
  
  Using the proof of Lemma \ref{P3.2} (i). Step 2, one sees that $\lim _{n \rightarrow+\infty} \int_{\Omega} T_1\left(u-T_n(u)\right) f d x=0$ and $\lim _{n \rightarrow+\infty} \int_{\Omega} F . D\left[T_1\left(u-T_n(u)\right)\right] d x=0$ so that we get (\ref{1.9}).

  \section{ Proofs of Theorem \ref{1.2} and Theorem \ref{1.3}}

  Let us first prove Theorem \ref{1.2}.
  \begin{proof}

   Let us consider the function $h_n, n>0$, defined on $\mathbb{R}$ by $h_n(r)=\inf \{1,(n+1-$ $\left.|r|)^{+}\right\}$. If $(u, w)$ is a solution of $(E,\mu)$. we take $h=h_n$ in (\ref{1.8}). We have for any $\varphi \in W_0^{1, {\overrightarrow{p}}}(\Omega) \cap$ $L^{\infty}(\Omega)$
  \begin{equation}\label{4.1}
  \int_{\Omega} a(x, D u) \cdot D\left(h_n(u) \varphi\right) d x+\int_{\Omega} w h_n(u) \varphi d x+\int_{\Omega} h_n(u) \varphi d v=\int_{\Omega} h_n(u) \varphi d \mu.
  \end{equation}
  
  Since $D\left(h_n(u) \varphi\right)=h_n(u) D \varphi+h_n^{\prime}(u) \varphi D u$ it follows from (\ref{4.1})
  
  \begin{equation}\label{4.2}
  \begin{aligned}
  & \int_{\Omega} h_n(u) a(x, D u) \cdot D \varphi d x+\int_{\Omega} h_n^{\prime}(u) \varphi a(x, D u) \cdot D u d x+\int_{\Omega} w h_n(u) \varphi d x+\int_{\Omega} h_n(u) \varphi d v \\
  & \quad=\int_{\Omega} h_n(u) \varphi d \mu .
  \end{aligned}
  \end{equation}
  
  We have $h_n(u) \longrightarrow 1$ a.e. in $\mathbb{R}$ as $n \longrightarrow+\infty$. Excepted the second term, all the terms in (\ref{4.2}) pass to the limit by Lebesgue dominated convergence theorem and when $n \longrightarrow+\infty$, we get
  
  \begin{equation}\label{4.3}
  \int_{\Omega} a(x, D u) \cdot D \varphi d x+\limsup _{n \rightarrow+\infty} \int_{\Omega} h_n^{\prime}(u) \varphi a(x, D u) \cdot D u d x+\int_{\Omega} w \varphi d x+\int_{\Omega} \varphi d v=\int_{\Omega} \varphi d \mu .
  \end{equation}
  
  For the second term in (\ref{4.3}), we have
  $$
  \begin{aligned}
  & \left|\int_{\Omega} h_n^{\prime}(u) \varphi c(x, D u) \cdot D u d x\right| \\
  & \leq \int_{[n \leq|u| \leq n+1]}\left|h_n^{\prime}(u) \varphi a(x, D u) \cdot D u\right| d x \\
  & \leq\|\varphi\|_{\infty} \int_{[n \leq|u| \leq n+1]}|a(x, D u) \cdot D u| d x \\
  & \leq \Lambda\|\varphi\|_{\infty} \sum_{i=1}^{N}\int_{[n \leq|u| \leq n+1]}\left(j_1(x)|D u|+|D u|^{p_i}\right) d x \\
  & \leq \Lambda\|\varphi\|_{\infty}\left(\left\|j_1\right\|_{L \overrightarrow{p^{\prime}}(\Omega)}\left(\sum_{i=1}^{N}\int_{[n \leq|u| \leq n+1]}|D u|^{p_i} d x\right)^{\frac{1}{p_i}}+\sum_{i=1}^{N}\int_{[n \leq|u| \leq n+1]}|D u|^{p_i} d x\right) \\
  & \longrightarrow 0 \text { as } n \longrightarrow+\infty \text { (thanks to (\ref{1.9})). } \\
  &
  \end{aligned}
  $$
  Now, we replace $\varphi$ by $T_k(u-\xi)$ in (\ref{4.3}) to get
  \begin{equation}\label{4.4}
  \int_{\Omega} a(x, D u) \cdot D T_k(u-\xi) d x+\int_{\Omega} w T_k(u-\xi) d x+\int_{\Omega} T_k(u-\xi) d v=\int_{\Omega} T_k(u-\xi) d \mu .
  \end{equation}
  
  Note that, since $\xi \in \operatorname{dom} \beta$.

  $$
  \begin{aligned}
  \int_{\Omega} T_k(u-\xi) d v & =\int_{\Omega} T_k(u-\xi) d v^{+}-\int_{\Omega} T_k(u-\xi) d v^{-} \\
  & =\int_{\substack{[u=M]}} T_k(u-\xi) d v^{+}-\int_{[u=m]} T_k(u-\xi) d v^{-} \\
  & \geq 0 ;
  \end{aligned}
  $$
  then (\ref{1.10}) follows from (\ref{4.4}).
  \end{proof}
  
  \begin{remark}

  The uniqueness of a solution in the sense of Theorem \ref{1.1} and Corollary \ref{C1.1} is not clear in general. Thanks to \cite{ref:33}, the uniqueness of a solution in the sense of Corollary \ref{C1.1} holds to be true for the so-called comparable solutions. That is, any two solutions $\left(u_1, w_1\right)$ and $\left(u_2, w_2\right)$ such that the difference $u_1-u_2$ is bounded. We'll not abort this question in this paper and refer the reader to the papers \cite{ref:33} and \cite{ref:193} for more details in this direction.
  \end{remark}
  Now, let us prove the uniqueness of the solution for $(E,\mu)$ when $-\infty<m \leq 0 \leq M<\infty$. Suppose that $\left(u_1, w_1\right),\left(u_2, w_2\right)$ are two solutions of $(E,\mu)$. For $u_1$. we choose $\xi=u_2$ as test function in (\ref{1.10}). we have
  $$
  \int_{\Omega} a\left(x, D u_1\right) \cdot D T_k\left(u_1-u_2\right) d x+\int_{\Omega} w_1 T_k\left(u_1-u_2\right) d x \leq \int_{\Omega} T_k\left(u_1-u_2\right) d \mu .
  $$
  
  Similarly we get for $u_2$
  
  $$
  \int_{\Omega} a\left(x, D u_2\right) \cdot D T_k\left(u_2-u_1\right) d x+\int_{\Omega} w_2 T_k\left(u_2-u_1\right) d x \leq \int_{\Omega} T_k\left(u_2-u_1\right) d \mu .
  $$
  
  Adding these two last inequalities yields
  
  \begin{equation}\label{4.5}
  \int_{\Omega}\left(a\left(x, D u_1\right)-a\left(x, D u_2\right)\right) \cdot D T_k\left(u_1-u_2\right) d x+\int_{\Omega}\left(w_1-w_2\right) T_k\left(u_1-u_2\right) d x \leq 0 .
  \end{equation}
  
  For any $k>0$, from (\ref{4.5}) it yields
  \begin{equation}\label{4.6}
  \int_{\Omega}\left(a\left(x, D u_1\right)-a\left(x, D u_2\right)\right) \cdot D T_k\left(u_1-u_2\right) d x=0 .
  \end{equation}
  
  From (\ref{4.6}), it follows that there exists a constant $c$ such that $u_1-u_2=c$ a.e. in $\Omega$. Using the fact that $u_1=u_2=0$ on $\partial \Omega$ we get $c=0$. Thus, $u_1=u_2$ a.e. in $\Omega$. At last, let us see that $w_1=w_2$ a.e. in $\Omega$ and $v_1=v_2$. Indeed for any $\varphi \in \mathcal{D}(\Omega)$, taking $\varphi$ as test function in (\ref{1.11}) for the solutions $\left(u_1, w_1\right)$ and $\left(u_1, w_2\right)$. after subtraction of these equalities we get
  
  $$
  \int_{\Omega}\left(w_1-w_2\right) \varphi d x+\int_{\Omega} \varphi d\left(v_1-v_2\right)=0 .
  $$
  
  Hence
  
  $$
  \int_{\Omega} w_1 \varphi d x+\int_{\Omega} \varphi d v_1=\int_{\Omega} w_2 \varphi d x+\int_{\Omega} \varphi d v_2
  $$
  
  Therefore
  
  $$
  w_1 \mathcal{L}^N+v_1=w_2 \mathcal{L}^N+v_2 .
  $$

  Since the Radon-Nikodym decomposition of a measure is unique, we get $w_1=w_2$ a.e. in $\Omega$ and $v_1=v_2$.
  
  To complete the proof of Theorem (\ref{1.3}), it remains to show that (\ref{1.12}) and (\ref{1.13}) hold. To this aim. let us prove the following result.
  \begin{lemma}\label{L4.1}

   Let $\eta \in W_0^{1, {\overrightarrow{p}}}(\Omega), Z \in \mathcal{M}_b^{\overrightarrow{p}}(\Omega)$ and $\lambda \in \mathbb{R}$ be such that
  \begin{equation}\label{4.7}
  \left\{\begin{array}{l}
  \eta \leq \lambda \text { a.e. in } \Omega(\text { resp. } \eta \geq \lambda), \\
  Z=-\operatorname{div} a(x, D \eta) \quad \text { in } \mathcal{D}^{\prime}(\Omega) .
  \end{array}\right.
  \end{equation}
  
  Then
  
  \begin{equation}\label{4.8}
  \int_{[\eta=\lambda]} \xi d Z \geq 0
  \end{equation}
  
  (resp.)
  
  \begin{equation}\label{4.9}
  \int_{[\eta=\lambda]} \xi d Z \leq 0
  \end{equation}
  
  for any $\xi \in C_c^1(\Omega) . \xi \geq 0$.
  \end{lemma}
  \begin{proof}

   The proof of this Lemma follows the same steps of \cite{ref:201}. For seek of completeness, let us give the arguments. For any $n \geq 1$, we set $\varphi_n(r)=\inf \left\{1,(n r-n \lambda+1)^{+}\right\}, \forall r \in \mathbb{R}$. Since $Z \in \mathcal{M}_b^{\overrightarrow{p}}(\Omega), \varphi_n(\eta) \longrightarrow$ $\chi_{\{\eta=\lambda\}}$ quasi everywhere, and since $Z$ is diffuse the convergence is also $Z$-a.e. Then for any $\xi \in C_c^1(\Omega)$. $\xi \geq 0$, we have
  $$
  \begin{aligned}
  \int_{[\eta=\lambda]} \xi d Z & =\lim _{n \rightarrow+\infty} \int_{\Omega} \xi \varphi_n(\eta) d Z \\
  & =\lim _{n \rightarrow+\infty} \int_{\Omega} a(x, D \eta) \cdot D\left[\xi \varphi_n(\eta)\right] d x \\
  & \geq \lim _{n \rightarrow+\infty} \int_{\Omega} \varphi_n(\eta) a(x, D \eta) \cdot D \xi d x
  \end{aligned}
  $$
  Furthermore
  $$
  \left|\int_{\Omega} \varphi_n(\eta) a(x, D \eta) \cdot D \xi d x\right| \leq\|D \xi\|_{\infty} \int_{\left[\lambda-\frac{1}{n} \leq \eta \leq \lambda\right]}|a(x, D \eta)| d x \rightarrow 0 \text { as } n \rightarrow+\infty
  $$
  
  This gives (\ref{4.8}). The proof of (\ref{4.9}) follows the same way by letting $\tilde{\eta}=-\eta, \tilde{\lambda}=-\lambda, \tilde{Z}=-Z$ and $\tilde{a}(x, z)=-a(x,-z)$.
  \end{proof}
  
  Coming back to the proof of (\ref{1.12}) and (\ref{1.13}), we see that, since
  $$
  v=\operatorname{div} a(x, D u)-w \mathcal{L}^N+\mu
  $$

  we have
  $$
  \mu-v-w \mathcal{L}^N=-\operatorname{div} a(x, D u) .
  $$
  By Lemma \ref{L4.1}, for any $\xi \in C_c^1(\Omega), \xi \geq 0$, we have
  $$
  \int_{[u=M]} \xi d v^{+} \leq \int_{[u=M]} \xi d \mu-\int_{[u=M]} \xi w d x
  $$
  and
  $$
  -\int_{[u=m]} \xi d v^{-} \geq \int_{[u=m]} \xi d \mu-\int_{[u=m]} \xi w d x .
  $$
  The first inequality implies that
  $$
  \int_{\Omega} \xi d v^{+} \leq \int_{\Omega} \xi d \mu \lfloor[u=M]-\int_{\Omega} \xi w \chi_{[u=M]} d x .
  $$
  Consequently (\ref{1.12}) holds. Similarly we get (\ref{1.13}).


\section{References}



\begin{thebibliography}{9}

\item[] \medskip\textbf{Journal article:}\smallskip



\bibitem{ref:L1} Y. Akdim, M. El Ansari, S. L. Rhali, Existence of a renormalized solution for some nonlinear anisotropic elliptic problems, Gulf Journal of Mathematics. \textbf{6}, 170-181 (2018)


\bibitem{ref:L2} Y. Akdim, M. El Ansari, S. L. Rhali, Solvability of some Stefan type problems with $L^1$-data, published in Rendiconti di Matematica e delle sue Applicazioni. (2023)


\bibitem{ref:197} Y. Akdim , R. Elharch, M. C. Hassib  and S. Lalaoui Rhali, Capacity and anisotropic sobolev spaces with zero boundary values, Nonlinear Dynamics and Systems Theory. \textbf{22}, 1-12 (2022) 

\bibitem{ref:203} B. Andreianov, K. Sbihi, P. Wittbold, On uniqueness and existence of entropy solutions for a nonlinear parabolic problem with absorption, J. Evol. Equ. \textbf{8}, 449-490 (2008) 

\bibitem{ref:3}  F. Andreu, N. Igbida, J. M. Maz\'{o}n and J.
Toledo, $L^{1}$ existence and uniqueness results for quasi-linear elliptic
equations with nonlinear boundary conditions. Ann. Inst. H. Poincar\'{e}.
Anal. Non Lin\'{e}aire. \textbf{24}, 61-89 (2007)

\bibitem{ref:201} F. Andreu, N. Igbida, J.M. Maz\'{o}n, Obstacle problems for degenerate elliptic equation with nonhomogeneous nonlinear boundary conditions, Math. Models Methods Appl. Sci. \textbf{18}, 1869-1893 (2008) 


\bibitem{ref:196} F. Andreu, N. Igbida, J. M. Maz\'{o}n and  J. Toledo, Degenerate elliptic equations with nonlinear boundary conditions and measures data, Ann. Scuola Norm. Sup. Pisa Cl. Sci. \textbf{8}, 1-37 (2009) 

 \bibitem{ref:4}   P. B\'{e}nilan, L. Boccardo, T. Gallou$\ddot{e}$t, R.
  Gariepy, M. Pierre, and J.L.V\'{a}zquez: An $L^{1}$- theory of existence and
  uniqueness of solutions of nonlinear equations, Ann. Scuola Norm. Sup. Pisa,
  Cl. Sci. \textbf{22}, 241-273 (1995)

\bibitem{ref:194} P. B\'{e}nillan and H. Brezis, Nonlinear problems related to the Thomas-Fermi equation, dedicated to Philippe B\'{e}nilan, J. Evol. Equ. \textbf{3}, 673-770 (2003)

\bibitem{ref:200} A. Bensoussan, L. Boccardo, F. Murat, On a nonlinear partial differential equation having natural growth terms and unbounded
solution, Ann. Inst. H. Poincaré Sect.  \textbf{5}, 347-364 (1988) 


 \bibitem{ref:27} L. Boccardo, J.I. D\'{\i}az, D. Giachetti,  F. Murat, Existence of a solution for a weaker form of a nonlinear elliptic equation, Recent advances
 in nonlinear elliptic and parabolic problems (Nancy 1988), Pitman Res.
 Notes Mat. Ser.  Longman Sci. Tech, Harlow. \textbf{208}, 229-246 (1989)
 
\bibitem{ref:190}L. Boccardo,  T. Gallou$\ddot{e}$t ,  Nonlinear elliptic equations with right hand
 side measures, Comm. Partial Differential Equations. \textbf{17}, 641-655  (1992)


\bibitem{ref:192} L. Boccardo, T. Gallou\"{e}t and L. Orsina, Existence and Uniqueness of Entropy Solutions for Nonlinear Elliptic Equations with Measure Data, Ann. Inst. H. Poincar\'{e} Anal.
Non Lin\'{e}aire. \textbf{13}, 539-551 (1996)



\bibitem{ref:204} H. Br\'{e}zis, A.C. Ponce, Reduced measures for obstacle problems, Adv. Differential Equations. \textbf{10}, 1201-1234 (2005) 


\bibitem{ref:199} J. Brooks, R. Chacon, Continuity and compactness of measures, Adv. Math. \textbf{37}, 16-26 (1980) 

\bibitem{ref:205} Y. Chen, S. Levine, and M. Rao, Variable exponent, linear growth functionals in image restoration, SIAM Journal on Applied Mathematics.  \textbf{66}, 1383-1406 (2006)


\bibitem{ref:191} A. Dall'Aglio, Approximated solutions of equations with $L^1$-data. Application to the H-convergence of parabolic quasi-linear equations, Ann. Mat. Pura Appl. \textbf{170}, 207-240 (1996) 


 \bibitem{ref:33}G. Dal Maso, F. Murat, L. Orsina, A. Prignet, Renormalized solutions of elliptic
 equations with general measure data, Ann. Scuola Norm. Sup. Pisa Cl. Sci. \textbf{28}, 741-808  (1999)
 
 

 \bibitem{D1}  R.J. DiPerna, P-L, Lions, On the Cauchy
 problem for Boltzmann equations: global existence and weak stability, Ann.
 of Math. \textbf{130}, 321-366 (1989)


 \bibitem{ref:16}   I. Fragal\`{a}, F. Gazzola, B. Kawohl,
 Existence and nonexistence results for anisotropic quasilinear elliptic
 equations, Ann. Inst. H. Poincar\'{e} Anal. Non Lin\'{e}aire. \textbf{21}, 15-734 (2004)
 
 
\bibitem{ref:198} M. Fukushima, K. Sato, S. Taniguchi,On the closable part of pre-Dirichlet forms and the fine support of the underlying measures, Osaka J.
Math. \textbf{28}, 517-535 (1991)



\bibitem{ref:193} O. Guib\'{e}, Remarks on the uniqueness of comparable renormalized solutions of elliptic equations with measure data, Ann.
Mat. Pura Appl. \textbf{180}, 441-449 (2002) 


\bibitem{ref:195} N. Igbida, S. Ouaro and S. Soma, Elliptic problem involving diffuse measure data, J. Differential Equations. \textbf{253}, 3159-3183 (2012) 



    \bibitem{ref:12}   M. Troisi, Teoremi di inclusione per spazi di
 Sobolev non isotropi, Ricerche Mat. \textbf{18}, 3-24 (1969)


\bibitem{ref:202} P. Wittbold, Nonlinear diffusion with absorption, Potential Anal. \textbf{7}, 437-457 (1997) 










\item[] \medskip\textbf{Book:}\smallskip

\bibitem{ref:8}  H. Br\'{e}zis. Op\'{e}rateurs Maximaux
 Monotones. North-Holland, Amsterdam, (1973)
 
 
  \bibitem{ref:15}   J. L. Lions, Quelques m\'{e}thodes de r\'{e}%
  solution des probl\`{e}mes aux limites non lin\'{e}aires, Dunod, Paris, (1969)

 
 \bibitem{ref:17}   M. R\r{u}\v{z}i\v{c}ka, Electrorheological Fluids:
  Modelling and Mathematical Theory, Springer, Berlin, (2000)
 
 
\end{thebibliography}
\end{document}